\documentclass[10pt,a4paper]{article}

\usepackage{tikz}
\usepackage{tikz-cd}
\usepackage{amssymb, url}
\usepackage{amsthm}
\usepackage{amsmath}
\usepackage{amsfonts}
\usepackage{enumitem}
\usepackage[colorlinks=true]{hyperref} 
\hypersetup{urlcolor=blue,linkcolor=black,citecolor=black,colorlinks=true}
\usepackage[symbol=$\uparrow~$,numberlinked=true]{footnotebackref}
\usepackage[all]{xy}
\usepackage{verbatim}
\usepackage{soul}

\theoremstyle{plain} % the usual style for theorems
\newtheorem{Theorem}{Theorem}[section]
\newtheorem{Lemma}[Theorem]{Lemma}
\newtheorem{Corollary}[Theorem]{Corollary}
\newtheorem{Proposition}[Theorem]{Proposition}
\newtheorem*{Proposition*}{Proposition}
\newtheorem*{TheoremA}{Theorem A}
\newtheorem*{TheoremB}{Theorem B}
\newtheorem*{TheoremC}{Theorem C}
\newtheorem*{TheoremD}{Theorem D}

\theoremstyle{definition}

\newtheorem{Definition}[Theorem]{Definition}
\newtheorem{Example}[Theorem]{Example}

\newtheorem*{Question*}{Question}

\theoremstyle{remark}
\newtheorem{Remark}[Theorem]{Remark}
\newtheorem*{Remark*}{Remark}

\newenvironment{ack}{\bigskip \noindent {\bf Acknowledgement}\it}

\newcommand{\Spec}{\operatorname{Spec}}

\newcommand{\Div}{\operatorname{Div}}
\newcommand{\tr}{\operatorname{tr}}

\newcommand{\Hom}{\operatorname{Hom}}
\newcommand{\End}{\operatorname{End}}
\newcommand{\Aut}{\operatorname{Aut}}
\newcommand{\id}{\operatorname{id}}

\newcommand{\sgn}{\operatorname{sgn}}

\newcommand{\Gal}{\operatorname{Gal}}
\newcommand{\Frob}{\operatorname{Frob}}

\newcommand{\GL}{\operatorname{GL}}

\newcommand{\Lie}{\operatorname{Lie}}

\newcommand{\CI}{\mathbb{C}_{\infty}}

\newcommand{\bC}{\mathbb{C}}

\newcommand{\bF}{\mathbb{F}}
\newcommand{\bG}{\mathbb{G}}

\newcommand{\bP}{\mathbb{P}}

\newcommand{\bT}{\mathbb{T}}

\newcommand{\bZ}{\mathbb{Z}}

\newcommand{\cO}{\mathcal{O}}

\newcommand{\fd}{\mathfrak{d}}

\newcommand{\fj}{\mathfrak{j}}

\newcommand{\fm}{\mathfrak{m}}

\newcommand{\fp}{\mathfrak{p}}

\newcommand{\fv}{\mathfrak{v}}

\newcommand{\fP}{\mathfrak{P}}
\newcommand{\fD}{\mathfrak{D}}

\newcommand{\CC}{\mathsf{C}}
\newcommand{\CD}{\mathsf{D}}
\newcommand{\Cc}{\mathsf{c}}

\newcommand{\degp}{d_{\fp}}
\newcommand{\ps}[1]{[\![#1]\!]}  % brackets for power series 
\newcommand{\ls}[1]{(\!(#1)\!)}% brackets for laurent series
\newcommand{\cs}[1]{\langle #1 \rangle} % brackets for convergent series

\newcommand{\fsf}{\mathfrak{sf}}
\newcommand{\mot}{\mathsf{M}}
\newcommand{\FI}{\bF_{\!\infty}}
\newcommand{\Fp}{\bF_{\!\fp}}
\newcommand{\Fq}{\bF_{\!q}}
\newcommand{\Fbar}{\overline{\bF}}
\newcommand{\Fbarq}{\overline{\bF}_{\!q}}

\newcommand{\up}{u_\fp}

\numberwithin{equation}{subsection}
\title{ 
\bf Special Functions and Gauss-Thakur Sums in Higher Rank and Dimension}

\author{Quentin Gazda\thanks{Current address: 
Univ Lyon, Universit\'e Jean Monnet Saint-\'Etienne, CNRS UMR 5208, Institut Camille Jordan, F-42023 Saint-\'Etienne, France},\\ and Andreas Maurischat\thanks{Current address: Interdisciplinary Center for Scientific Computing, University of Heidelberg, Im Neuenheimer Feld 205, G-69120 Heidelberg, Germany }}
\date{August 21, 2020}
\begin{document}

\maketitle

%\tableofcontents

%Our comments, temporary
%\input{Text/Comments}

%Intro, to discuss at the end
%\input{Text/Introduction}
\begin{abstract}
Anderson generating functions have received a growing attention in function field arithmetic in the last years.
Despite their introduction by Anderson in the 80s where they were at the heart of comparison isomorphisms, further important applications e.g. to transcendence theory have only been discovered recently. The Anderson-Thakur special function interpolates L-values via Pellarin-type identities, and its values at algebraic elements recover Gauss-Thakur sums, as shown by Angl\`es and Pellarin.
For Drinfeld-Hayes modules, generalizations of Anderson generating functions have been introduced by Green-Papanikolas and -- under the name of ``special functions'' -- by Angl\`es-Ngo Dac-Tavares Ribeiro.

In this article, we provide a general construction of special functions attached to any Anderson A-module. We show direct links of the space of special functions to the period lattice, and to the Betti cohomology of the A-motive. 

We also undertake the study of Gauss-Thakur sums for Anderson A-modules, and show that the result of Angl\`es-Pellarin relating values of the special functions to Gauss-Thakur sums holds in this generality.
\end{abstract}

\section{Introduction}

%\subsection{Motivation and Results}

\paragraph{Special functions}

The Carlitz module $\CC$ over the rational function field $\Fq(\theta)$ serves as a function field analogue of the multiplicative group $\bG_m$ over number fields. For example, it is used to investigate class field theory for  $\Fq(\theta)$, it comes with an exponential function $\exp_{\CC}$, and the kernel of this exponential -- called the period lattice of $\CC$ -- is an $\Fq[\theta]$-lattice of rank $1$ generated by the so called Carlitz period
\[ \tilde{\pi}=\sqrt[q-1]{-\theta}\cdot \theta \prod_{j=1}^{\infty} \left(1 - \theta^{1-q^j}\right)^{-1} \in \Fq\ls{\tfrac{1}{\theta}}(\sqrt[q-1]{-\theta}), \]
which plays an analogous role as $2\pi i$ does in the classical setting. 

\medskip

The Anderson-Thakur special function $\omega$ plays a central role for the Carlitz module, and is the most basic example of special functions that we are discussing in this paper. From $\omega$, one can recover the Carlitz period, define the "Betti homology" of $\CC$, and calculate special $L$-values (see \cite{pellarin}). These applications don't have clear analogue in number fields. 

Let us introduce some notations. Let $\CI$ denote the completion of an algebraic closure of $\Fq\ls{\tfrac{1}{\theta}}$ with respect to the absolute value given by $|\theta|=q$. 
The Carlitz module $\CC$ is the additive group $\bG_a(\CI)=\CI$ endowed with an $\Fq[t]$-action $\Cc$ given by $\Cc_t=\theta+\tau$ where $\tau$ is the $q$-power Frobenius map on $\CI$.
The Carlitz exponential is the unique $\Fq$-linear map
$\exp_\CC:\CI\to \bG_a(\CI)=\CI$ such that
\[ \exp_\CC(a(\theta) x)=\Cc_a (\exp_\CC(x)) \qquad \forall x\in \CI,\ a\in \Fq[t] \]
(or equivalently just $ \exp_\CC(\theta x)=\Cc_t (\exp_\CC(x)) \, \forall x\in \CI$)
and whose derivative is the identity. 

Let $\CI\cs{t}$ denote the Tate-algebra over $\CI$, namely the subring of power series in $\CI\ps{t}$ consisting of those power series which converge on the unit disc. One way to define $\omega(t)$ is via the following element in $\CI\cs{t}$ (see e.g.~\cite{el-guindy})):
%\[ \omega(t) := \sqrt[q-1]{-\theta} \prod_{j \geq 0} \left(1 - \frac{t}{\theta^{q^j}}\right)^{-1}. \]
%By choosing the same $(q-1)$-th root of $-\theta$ as for $\tilde{\pi}$ above, one gets the relation
%\[ \tilde{\pi} = - \text{res}_{t = \theta} \omega = -\lim_{t \rightarrow \theta} (t - \theta) \omega(t). \]
%On the other hand,  using the Carlitz exponential $\exp_{\CC}$, one can obtain $\omega$ from $\tilde{\pi}$ via
\[  \omega(t)= \sum_{n=0}^\infty \exp_{\CC}\left(\frac{\tilde{\pi}}{\theta^{n+1}}\right)t^n. \]
The function $\omega$ was expressed succinctly by Angl\`es and Pellarin in \cite[Sect.~2.2]{angles} as
\begin{equation}\label{eq:pi-to-omega}
\omega(t)= \exp_{\CC}\left(\frac{\tilde{\pi}}{\theta-t}\right)
\end{equation}
by extending the Carlitz exponential $t$-linearly continuously to $\CI\cs{t}$.
In a concurrent manner, they extended the Carlitz action $\Cc$ to a $t$-linear continuous action on $\CI\cs{t}$, and readily obtained that
\[ (\Cc_t-t)(\omega)=\exp_\CC(\tilde{\pi})=0. \]
In words, $\omega$ is a ``$\CI\cs{t}$-point of the Carlitz module $\CC$'' for which the two $\Fq[t]$-actions -- the one via $\Cc$ and the other via multiplication by $t$ -- coincide. 
As  the Anderson-Thakur special function $\omega(t)$ is an invertible element in $\CI\cs{t}$, it even generates the $\Fq[t]$-module 
\[ \fsf(\CC):= \{ h\in \CC(\CI\cs{t}) \mid \Cc_t (h)=t\cdot h \} \]
of those functions on which both $\Fq[t]$-actions coincide. 
%
%\medskip

Another description of $\omega$, namely as
\[  \omega(t) = \sqrt[q-1]{-\theta} \prod_{j=0}^{\infty} \left(1 - \frac{t}{\theta^{q^j}}\right)^{-1} ,\]
stems from the interpretation of $\omega^{-1}$ as a rigid analytic trivialization of the Carlitz motive. %$\mot(\CC)=\Hom(\CC,\bG_a)\cong \CI\{\tau\}$ 
This property boils down to the equation
\[  \omega^{(1)}(t) = (t-\theta)\omega(t), \]
where $()^{(1)}$ denotes the usual Frobenius twist on $\CI\cs{t}$, namely the $t$-linear continuous extension of $\tau$ to $\CI\cs{t}$. It readily implies the equality:
\begin{equation}\label{eq:drinfeld-solution-tau-solution}
 \{ h\in \CC(\CI\cs{t}) \mid \Cc_t (h)=t\cdot h \}\,\, =\,\, \{ h\in \CI\cs{t} \mid h^{(1)}=(t-\theta)h \}. 
\end{equation} 
This correspondence between solutions of the $t$-action equation (the left hand side of Equation \eqref{eq:drinfeld-solution-tau-solution}) and solutions of a $\tau$-difference equation (the right hand side of Equation \eqref{eq:drinfeld-solution-tau-solution}) can be traced back to the work of Sinha \cite{sinha}. It has also been observed in various other particular situations, although sometimes only over the field of fraction of $\CI\cs{t}$ or its analogue in the $A$-module case, respectively. For example, this duality appears for Carlitz tensor powers in \cite{andersonthakur}, for Drinfeld $\Fq[t]$-modules in \cite[\S 4.2]{pellarin08} (see also \cite[Sect.~2]{el-guindy}), for Drinfeld-Hayes $A$-modules in \cite{green}, and in \cite{tuan}, as well as for general Anderson $t$-modules in \cite[Sect.~3.2]{maurischat}.
We will show in Theorem \ref{funceq} that this duality holds in great generality for any Anderson $A$-module (see below).

\medskip

Let us first recall the case of Drinfeld-Hayes modules $\CD$ over some coefficient ring $A$. Green-Papanikolas \cite[Sect.~4]{green} (for $A$ the ring of functions on an elliptic curve regular outside infinity) and Angl\`es-Ngo Dac-Tavares Ribeiro \cite[Rmk.~3.10]{tuan} (general $A$) have noticed that Equation \eqref{eq:drinfeld-solution-tau-solution} holds where $(t-\theta)$ is replaced by Thakur's shtuka function of $\CD$, and $\CI\cs{t}$ by the affinoid algebra $\bT=A\hat{\otimes}_{\Fq}\CI$ with twist $()^{(1)}$ induced by the identity on $A$ and $\tau$ on $\CI$. Their methods rely on the Drinfeld/Mumford correspondence for rank one Drinfeld-modules.

Our approach recovers this equation from a more general framework. 
Let $(E,\phi)$ be an Anderson $A$-module over $\CI$, let $\mot=\Hom_{\bF,\CI}(E,\bG_{a,\CI})$ be its $A$-motive (which may be not finitely generated over $A\otimes_{\Fq} \CI$, if $E$ is not abelian), and $\mot_{\bT}:=\mot\otimes_{A\otimes \CI} \bT$ the scalar extension to $\bT=A\hat{\otimes}_{\Fq}\CI$.
We define the space of special functions to be the $A$-module
\[  \fsf(E) :=\{ h\in E(\bT) \mid \forall a\in A: \phi_a (h)=a\cdot h\} \]
similar as for the Carlitz case (for the precise definition of the notion see Section \ref{sec:special-functions}).

\begin{TheoremA}[Theorem \ref{funceq}]
There is a natural $A$-linear isomorphism 
\[\fsf(E)\longrightarrow \Hom_{\bT}^\tau(\mot_{\bT},\bT),\] where the latter denotes the $A$-module of $\tau$-equivariant $\bT$-linear homomorphisms. 
If $E$ is abelian and uniformizable, this induces an isomorphism of $A$-modules 
\[ \fsf(E)\longrightarrow \Hom_{A}(H_B(\mot),A), \]
where $H_B(\mot)=\{ m\in \mot_{\bT} \mid \tau_{\mot}(m)=m \}$ denotes the Betti homology of $\mot$.
\end{TheoremA}

The transition from this abstract isomorphism to descriptions as in Equation \eqref{eq:drinfeld-solution-tau-solution} can be given for abelian $A$-modules and is explained in Remark \ref{rem:explicit-systems-of-equations}. The rough idea is to choose a coordinate system for $E$, i.e.~an isomorphism $E(\CI)\cong \CI^d$ and express the special functions in these coordinates. On the right-hand side, one chooses an $A\otimes \CI$-basis of $\mot$ inducing an isomorphism $\mot\cong (A\otimes_{\Fq} \CI)^r$ of $A\otimes_{\Fq} \CI$-modules, and expresses the $\tau$-equivariance condition as a $\tau$-difference equation for vectors in $\bT^r$. $\mot$ might not be free as $A\otimes \CI$-module in which case one has to localize first in order to get a free module.
This is the reason why in some $A$-module settings the correspondence above was only obtained over the field of fractions of $\bT$.

In the context of Drinfeld-Hayes modules, Theorem $A$ allows to obtain previously known formulas (for instance Corollary \ref{start}) in a direct fashion, hence allowing to bypass theorems of Drinfeld and Mumford from the theory of shtukas (e.g. \cite[Chap.~6]{goss}). We develop this viewpoint in Section \ref{drinfeld-hayes}. This clarifies several technical issues and simplifies the presentation of Thakur's theory of shtuka functions introduced in \cite{thakur93}.

\paragraph{Relation to the period lattice}
Formula \eqref{eq:pi-to-omega} also hides an isomorphism of $\Fq[t]$-modules 
\begin{equation}\label{eq:lambda-to-sf}
\Lambda_\CC\longrightarrow \fsf(\CC), \quad \tilde{\pi}\longmapsto \omega, 
\end{equation}
where $\Lambda_\CC$ denotes the period lattice $\bF[\theta]\tilde{\pi}$ with $\Fq[t]$-action via
$a(t)\cdot \lambda=a(\theta)\lambda$. The same relation holds for any Drinfeld $\Fq[t]$-module $\CD$ between its period lattice $\Lambda_{\CD}$ and its space of special functions $\fsf(\CD)$ by using
the Anderson generating functions attached to the periods (see \cite[Prop.~6.2 \& Rem.~6.3]{el-guindy}). The second author extended it for general Anderson $t$-modules $E$ (of any rank and dimension) in \cite[Sect.~3.2]{maurischat}. From \cite[Sect.~4]{green}, one deduces that this still holds for particular Drinfeld-Hayes module when $A$ is the coefficient ring of an elliptic curve.
The relation $K\otimes_A \Lambda_E\cong K\otimes_{A}\fsf(E)$, where $K$ is the fraction field of $A$, has also been observed by Angl\`es-Ngo Dac-Tavares Ribeiro for Drinfeld-Hayes modules over general $A$ (see \cite[Sect.~3]{tuan}).

For general $A$, whether or not the space of special functions for an Anderson $A$-module $E$ is isomorphic to the period lattice $\Lambda_E$ solely depends on the ring $A$ (apart from special cases where they are ``accidentally'' isomorphic). More precisely, we show
\begin{TheoremB}[Theorem \ref{lattice}]
Let $u\in A$ such that $K/\Fq(u)$ is a finite separable extension, and let $\fd_{A/\Fq[u]}\subseteq A$ denote the different ideal of the extension. Then for any Anderson $A$-module $E$ there is an isomorphism (depending on the choice of $u$)
\[  \delta_u: \fd_{A/\Fq[u]}\cdot \Lambda_E \longrightarrow \fsf(E). \]
\end{TheoremB}

When $A$ is the coefficient ring of the one dimensional projective space or of an elliptic curve over $\Fq$, the different ideal $\fd_{A/\Fq[u]}$ is principal. The theorem explains why an isomorphism between the lattice and the special functions is obtained in the situations above. It also answers in the negative a question raised in \cite[end of Sect.~3.2]{tuan} whether the space of special functions is always free (see Corollary \ref{special-function-free}).

\paragraph{Gauss-Thakur sums}
Gauss-Thakur sums are the function field analogues of Gauss sums. They were introduced and studied by Thakur in a series of papers \cite{thakur88}, \cite{thakur91}, \cite{thakur91b}, \cite{thakur93}, \cite{thakur} where he established analogues of Stickelberger factorization, Hasse-Daven\-port and Gross-Koblitz theorem.
For a Drinfeld-Hayes module $\CD$, the Gauss-Thakur sums modulo some non-zero prime ideal $\fp\subset A$ are defined to be the sums in $\CI$
\[ g(\chi,\psi)=-\sum_{x\in (A/\fp)^{\times}} \chi(x)^{-1} \psi(x) \]
where $\chi:\bigl(A/\fp\bigr)^\times \to \Fbarq^\times$ is a group morphism and $\psi:A/\fp \to \CD(\CI)=\CI$ is a morphism of $A$-modules (see \cite{thakur91}).
We present a generalization of these Gauss-Thakur sums where the Drinfeld $A$-module of rank $1$ is replaced by an Anderson $A$-module $(E,\phi)$ of arbitrary dimension. As there is no canonical multiplication of $\Fbarq$ on the $\CI$-points $E(\CI)$, our Gauss-Thakur sums will be elements in the tensor product $\Fbarq \otimes_{\Fq} E(\CI)$:
\begin{equation}\label{thakurgausssum}
g(\chi,\psi):=-\sum_{x\in (A/\fp)^{\times}}{\chi(x)^{-1}\otimes \psi(x)}. \nonumber
\end{equation}
for a group morphism $\chi:\bigl(A/\fp\bigr)^\times \to \Fbarq^\times$ and a morphism of $A$-modules $\psi:A/\fp \to E(\CI)$. Since the image of a multiplicative character $\chi:\bigl(A/\fp\bigr)^\times \to \Fbarq^\times$ lies in the residue field $\Fp:= A/\fp$, we prefer to consider the sums inside $\Fp\otimes_{\Fq} E(\CI)$. Fixing $\chi$, we obtain the following properties:
\begin{enumerate}[label=$(\alph*)$]
\item (Lemma \ref{projection-on-gauss-space} \ref{item:c}) For any additive character $\psi$, the sum $g(\chi,\psi)$ satisfies
\begin{equation}\label{introeq}
(\chi(a)\otimes 1)g(\chi,\psi)=(1\otimes \phi_a)g(\chi,\psi) \quad \forall a \in A \nonumber
\end{equation}
where as usual $\chi$ is lifted to a map $\chi:A\to \Fp$ with additional
$\chi(a):=0$ if $a$ is in $\fp$.
\item (Lemma \ref{projection-on-gauss-space} \ref{item:c}) The $\Fp$-vector space 
\[ G(E,\chi):= \{ g\in \Fp \otimes_{\Fq} E(\CI) \mid \forall a \in A: (\chi(a)\otimes 1)g=(1\otimes \phi_a)g \} \]
is generated by Gauss-Thakur sums $\{g(\chi,\psi_i)\}_i$.
\item (Proposition \ref{nonzero} \ref{item:A}) In the case where the lift of $\chi$ to a map $A\to \Fp$ is not a homomorphism of $\Fq$-algebras, then $G(E,\chi)=0$, and in particular all $g(\chi,\psi)$ are zero.
\item (Proposition \ref{nonzero} \ref{item:B}) In the case where the lift of $\chi$ is a homomorphism of $\Fq$-algebras, then $g(\chi,\psi)$ is zero if and only if $\psi$ is zero. Furthermore, a family $\{g(\chi,\psi_i)\}_i$ is linearly independent over $\Fp$ if the family $\{\psi_i\}_i$ is. 
\end{enumerate}
Point $(d)$ generalizes Thakur's non-vanishing result \cite[Thm.~I.(3)]{thakur88}. Using the exponential map $\exp_E$, we can attach Gauss-Thakur sums to periods as follows. Let $\up\in \fp$ be a uniformizer for $\fp$, and choose an element $z_\fp$ in the fractional ideal $\fp^{-1}$ such that $z_\fp\up\equiv 1 \mod \fp$. Then for any period $\lambda\in \Lambda_E$, one has a homomorphism of $A$-modules $\psi_\lambda:A/\fp\to E[\fp]$ given by $\psi_\lambda(a)=\exp_E(a z_\fp\cdot \lambda)$ which does not depend on the choice of $z_\fp$, but only on $\up$. For a fixed character $\chi$, this induces a homomorphism of $A$-modules
\[   g_{\up}:\Lambda_E \longrightarrow G(E,\chi), \quad \lambda \longmapsto g(\chi,\psi_\lambda). \]
In Proposition \ref{uniformizable}, we show that $E$ is uniformizable if, and only if, every Gauss-Thakur sum is of the form $g(\chi,\psi_{\lambda})$ for some $\lambda\in \Lambda_E$, providing a new criterion for uniformizability.

Our main motivation for introducing these Gauss-Thakur sums lies in generalizing a remarkable relation proved by Angl\`es and Pellarin for the Carlitz module (see \cite[Thm.~2.9]{angles}):
\begin{equation}
\omega(\zeta)=\fp'(\zeta)g(\chi_{\zeta},\psi_{\tilde{\pi}}), \nonumber
\end{equation}
where $\fp(t)$ is a monic irreducible polynomial in $\Fq[t]$, $\zeta$ is one of its roots in $\Fbarq$, $\chi_{\zeta}$ is the group morphism given by evaluation at $\zeta$, and -- as a special case of the above notation -- $\psi_{\tilde{\pi}}$ is the morphism
$\psi_{\tilde{\pi}}(a)=\exp_{\CC}(a\fp^{-1}\tilde{\pi})$ for all $a\in A/\fp$. 

For a general Anderson $A$-module $E$, and $\fp$ a maximal ideal of $A$ with uniformizer $\up\in \fp$, we show:
\begin{TheoremC}[Theorem \ref{values}]
For every homomorphism of $\Fq$-algebras $\chi:A\to \Fp$, the following diagram of $A$-modules commutes: \\
\centerline{
\xymatrix@C+10mm{
\fd_{A/\bF[\up]}\cdot \Lambda_E \ar[r]^{\delta_{u_{\fp}}} \ar[dr]_{g_{\up}} &   \fsf(E) \ar[d]^{(\chi \hat{\otimes} \id)} \\
 & G(E,\chi)
}}
where $\delta_{\up}$ is the map of the above theorem for $u=\up$.
\end{TheoremC}
In the case of the Carlitz module $\CC$ (see Example \ref{example}), we take $\up=\fp(t)$ in $A=\Fq[t]$ to be the monic generator of the corresponding prime ideal. The different ideal $\fd_{A/\bF[\up]}$ is principal and generated by $\fp'(t)$, and $\delta_{\up}(\fp'(t)\cdot \tilde{\pi})$ is just the Anderson-Thakur function $\omega(t)$ (see Proposition \ref{prop:delta-always-the-same}). By taking the evaluative character $\chi_\zeta$ mapping $t$ to the root $\zeta$ of $\fp$, the commutativity of the diagram recovers Angl\`es and Pellarin's formula.
Whilst in \cite{angles}, it requires the computation of the sign of the Gauss-sum as given in \cite[Thm.~2.3]{thakur},
our theorem follows after a direct calculation. A juxtaposition of the two approaches might result in formulas for the sign of Gauss-Thakur sums. 

\medskip

In the introduction of \cite{angles}, such a relation gave birth to the appellation \textit{Universal Gauss-Thakur sum} for $\omega$, as all non-zero Gauss-Thakur sums can be recovered from values of $\omega$ at algebraic points. The above diagram thus extends the naming to any special function of an Anderson $A$-module. 

\medskip

As an application of our work, we use Green-Papanikolas Pellarin-type identities in \cite{green} to explain how results on Universal Gauss-Thakur sums may be helpful to compute function field special $L$-values. When $A$ is the coefficient ring of an elliptic curve, we prove:
\begin{TheoremD}[Theorem \ref{special-L-values-elliptic}]
Let $E$ be a sign-normalized Drinfeld-Hayes $A$-module over $H$ -- the Hilbert class field of $K$ -- for which the associated Drinfeld divisor is not supported on the unit disk $\operatorname{Spm}\bT$, and let $A^+$ denote the set of sign one elements in $A$. For any non-zero element $\tilde{\pi}_E$ in $\Lambda_E\subset \CI$, and for all non negative integers $n$, we have
\begin{equation}
L(\chi,q^n):=\sum_{a\in A^+}{\frac{\chi(a)}{a^{q^n}}} \in 
H(\chi) \cdot \frac{\tilde{\pi}_E^{q^n}}{g(\chi,\psi_{\tilde{\pi}_E})} \nonumber
\end{equation}
where $\chi:A\to \Fbarq$ is an $\Fq$-algebra homomorphism, and $g(\chi,\psi_{\tilde{\pi}_E})$ is regarded as a tensorless sum in $\CI$.
\end{TheoremD}
A similar identity can be derived for classical Dirichlet $L$-functions. However, the proof relies on their functional equation which has yet no analogue in function field arithmetic. The high level of similarity between those two formulas is then of a remarkable charm. 

\medskip

The paper is organized as follows. In Section \ref{sec:setting}, we introduce the basic notation used in this article, as well as the general Tate algebra $\bT$ and the ``$\bT$-points'' $E(\bT)$.
The special functions are introduced and investigated in Section \ref{sec:special-functions}, starting with the relation to the Betti realization of the associated motive which was implicitly used in the special cases. The relation to the period lattice, and in particular Theorem \ref{lattice}, is given in Subsection \ref{subsec:relation-to-lattice}, right after we discussed the neccessary properties of the relative different ideal in Subsection \ref{characterisation-of-different}. We end this section by reviewing the case of $A$-modules of rank and dimension $1$, together with Thakur's theory of shtuka functions.
Section \ref{sec:gauss-thakur-sums} is devoted to Gauss-Thakur sums for arbitrary Anderson $A$-modules, and its relation to the values of the special functions at algebraic points is given in Section \ref{sec:values}. In the last section, we apply our results to special values of Goss $L$-functions.

\begin{ack}
Both authors thank Rudy Perkins for a result on the tensor powers of the Carlitz module, that finally does not appear anymore in the paper, but led us on the right track. This work is part of the PhD thesis of the first author under the supervision of Federico Pellarin.
\end{ack}

%General notations, definitions
%\input{Text/Setting}
\section{Setting}\label{sec:setting}

Let $(C,\cO_C)$ be a smooth projective geometrically irreducible curve over a finite field $\bF$ with $q$ elements and characteristic $p$. We fix a closed point $\infty$ on $C$ and consider $A=H^0(C\setminus \{\infty\},\cO_C)$, the ring of rational functions on $C$ that are regular outside $\infty$. This is an algebra over the field~$\bF$. 

Let $K$ be the function field of $C$ (or equivalently, the fraction field of $A$). The degree and residue field of $\infty$ will be denoted by $d_{\infty}$ and $\FI$, respectively, and the associated norm will be $|\cdot|$. Further, let $\CI$ be the completion at a place above $\infty$ of an algebraic closure of $K$, to which we extend $|\cdot|$. We fix $L$ an intermediate field $K\subseteq L\subseteq \CI$, and denote the natural inclusion by $\ell:K\to L$.\\ By convention, every unlabeled tensor will be over $\bF$. 

On $\CI$ (and all its subrings) we let $\tau:\CI\to \CI$ be the $q$-power Frobenius map, and extend it $A$-linearly to $A\otimes \CI$, i.e.~$\tau(a\otimes x)=\left(a\otimes x^q\right)$ on elementary tensors in $A\otimes \CI$. As it is common in the $\bF[t]$-case, we also write $h^{(1)}$ instead of $\tau(h)$ for $h$ in $A\otimes \CI$, and call it the $\textit{Frobenius twist}$ of $h$.

Given a smooth commutative group scheme $E$ over $L$, we let $\Lie_E(L)$ be its Lie algebra, i.e.~its tangent space at the neutral element $e:\Spec L\to E$.
For any group scheme morphism $f:E\to E'$ over $L$, we denote the induced map on the Lie algebras by $\partial f:\Lie_E(L)\to \Lie_{E'}(L)$. 
We recall that an $\bF$-vector space scheme is a commutative group scheme equipped with compatible $\bF$-multiplication. The additive group scheme $\bG_{a,L}$ is naturally an $\bF$-vector space scheme. 

\medskip

Let $d$ be a positive integer. By an \textit{Anderson $A$-module $(E,\phi)$ over $L$ of dimension $d$} we mean the following (see \cite[Def.~5.1]{hartl}): 
\begin{enumerate}[label=(\roman*)]
\item $E$ is an $\bF$-vector space scheme over $L$ which is isomorphic to $\bG_{a,L}^d$, the $d^{th}$-power of the additive group scheme over $L$,
\item $\phi:A\to\End_{\bF,L}(E),a\mapsto \phi_a$ is a homomorphism of $\bF$-algebras into the ring of $\bF$-vector space scheme endomorphisms of $E$ over $L$,
\item for $a\in A$, the induced action $\partial\phi_a:\Lie_E(L)\to \Lie_E(L)$ of $\phi_a\in \End_{\bF,L}(E)$ satisfies that $(\partial \phi_a -\ell(a)\id)\in \End_L(\Lie_E(L))$ is nilpotent.
\end{enumerate}
The Anderson $A$-module $(E,\phi)$ is called \textit{abelian}, if furthermore
\begin{enumerate}[label=(\roman*), resume]
\item $\Hom_{\bF,L}(E,\bG_{a,L})$ is finitely generated as an $A\otimes L$-module where $a\in A$ is acting by composition with $\phi_a$ on the right and $l\in L$ is acting by composition on the left with multiplication by $l$.
\end{enumerate}

The $A\otimes L$-module $\Hom_{\bF,L}(E,\bG_{a,L})$ together with the $\tau$-semilinear action given by composition with the $q$-power Frobenius endomorphism on $\bG_{a,L}$ is called the $A$-motive of $E$, denoted by $\mot(E)$. If it is finitely generated as $A\otimes L$-module, it is even locally free (see \cite[Lemma 1.4.5]{anderson} for $A=\bF[t]$ which implies the general case), and its rank is also called the \textit{rank of $E$} (see \cite[Sect.~5]{hartl}).

As in \cite[Def.~3.1]{hartl}, we define the pullback of $\mot:=\mot(E)$ by $\tau$ as 
\begin{equation}
\tau^*\mot:= (A\otimes L)\otimes_{\tau,A\otimes L}\mot, \nonumber
\end{equation}
where $A\otimes L$ on the left-hand side of the tensor is given an $A\otimes L$-module structure via $\tau$. Then, the $\tau$-semilinear action of the Frobenius on $\mot$ induces an $A\otimes L$-linear map \[ \tau_{\mot}:\tau^*\mot\to \mot \]
 which is injective, but in general \textit{not} surjective.

\medskip

By abuse of notation, we will omit the $\phi$ of the Anderson $A$-module $(E,\phi)$ from now on, and will just write $ae$ instead of $\phi_a(e)$ for $a\in A$ and $e\in E(L)$, as well as $\partial a(x)$ instead of $\partial \phi_a(x)$ for $a\in A$ and $x\in \Lie_E(L)$. By the third condition on Anderson $A$-modules, the endomorphisms $\partial a$ are even automorphisms on $\Lie_E(L)$ for $a\neq 0$, and hence, we can define the automorphisms $\partial r\in \Aut_L(\Lie_E(L))$ for any $0\neq r=\frac{a}{b}\in K$, by $\partial r=(\partial b)^{-1} \partial a$.

For avoiding confusion, scalar multiplication of an element $r$ in $K\subseteq L$ with $x\in \Lie_E(L)$ will \emph{always} be written using the homomorphism $\ell$, i.e. as $\ell(r)x$. 

\medskip

To $E$, one naturally associates an $\bF$-linear map $\exp_E:\Lie_E(\CI)\to E(\CI)$ -- called the \textit{exponential function of} $E$ -- satisfying the properties $(a)$ and $(b)$ below:
\begin{enumerate}[label=(\roman*)]
\item For all $a$ in $A$ and all $x$ in $\Lie_E(\CI)$, $\exp_E((\partial a)x)=a\cdot \exp_E(x)$.
\item For any isomorphism $\kappa:E\stackrel{\sim}{\to} \bG_{a,L}^d$ of $\bF$-vector space schemes over $L$, there exists a sequence $(e_n^{\kappa})_{n\geq 0}$ of $d\times d$ matrices in $\CI$ with $e_0^{\kappa}=1$ for which the diagram 

\centerline{
\xymatrix@C+4mm{
\Lie_E(\CI) \ar[r]^{\exp_E} \ar[d]^{\partial \kappa} & E(\CI) \ar[d]^{\kappa} \\
\CI^d \ar[r]^{\exp_E^{\kappa}} & \CI^d
}}

\noindent in the category of $\bF$-vector spaces commutes, where $\exp_E^{\kappa}$ is computed by the converging series
\begin{equation}
\exp_E^{\kappa}(x)=x+e_1^{\kappa} \tau(x)+e_2^{\kappa} \tau^2(x)+... \quad \forall x\in \CI^d \nonumber
\end{equation}
where $\tau$ is applied coefficient-wise.
\end{enumerate}
We refer to \cite[Sect.~8.6]{boeckle} for a proof of its existence and uniqueness. Its kernel $\Lambda_E$ is usually called the \textit{period lattice of} $E$. 
%This is a discrete $A$-submodule of $\Lie_E(\CI)$, since $\exp_E^\kappa$ is an isometry in a small neighbourhood of $0$. 
One calls $E$ \textit{uniformizable} if the exponential map is surjective (see \cite[Def.~5.26]{hartl}).

\bigskip

\paragraph{Gauss completions}

%\subsection{Definition}
For defining the general Tate-algebra $\bT=A\hat{\otimes}\CI$, and the ``points'' $E(\bT)$, we need certain completions of tensor products.

\begin{Definition}
If $V$ is an $\bF$-vector space given with a non-archimedean norm $|\cdot|$ and $B$ is a countably dimensional $\bF$-algebra, we define a norm on $B\otimes V$ by
\begin{equation}
\|x\| :=\inf \left( \underset{i}{\max} |v_i| \right) \qquad \text{for}~x\in B\otimes V  \nonumber
\end{equation}
where the infimum is taken over all the representations of $x$ of the form $\sum_{i}{(b_i\otimes v_i)}$.
We denote by $B\hat{\otimes}V$ the completion of $B\otimes V$ with respect to this norm.

\end{Definition}

We gather some properties of this construction in the following proposition:
\begin{Proposition}\label{proptate}
Assume that $V$ is complete, and let $t=\{t_n\}_{n\geq 0}$ be any basis of $B$ as an $\bF$-vector space.
\begin{enumerate}[label=$(\arabic*)$]
\item \label{item:t1} For all $x=\sum_{n=0}^{\infty}{t_n\otimes v_n}\in B\otimes V$ (with only finitely many $v_n$ non-zero), one has $\|x\|=\max_i |v_i|$. 
\item \label{item:t2} For any $f$ in $B\hat{\otimes} V$, there exists a unique sequence $(v_n(f))_{n\geq 0}$ of elements in $V$ converging to zero such that the series
\begin{equation}\label{t-expansion}
\sum_{n=0}^{\infty}{t_n\otimes v_n(f)}
\end{equation}
converges to $f$ in $B\hat{\otimes} V$. The norm $\|f\|$ of $f$ is then given by the maximum of the $|v_n(f)|$ ($n\geq 0$). 
\end{enumerate}
\end{Proposition}

\begin{proof}
Part \ref{item:t1}:  
Let %$t=(t_n)_{n\geq 0}$ be a countable $\bF$-basis of $B$, and 
$\|x\|_t$ denote the maximum in norm of the coefficients of $x\in B\otimes V$ written in the basis $\{t_n\otimes 1\}_{n\geq 0}$. We claim that $\|x\|=\|x\|_t$: it is clear that $\|x\|\leq \|x\|_t$ so we prove the converse inequality. For $\varepsilon>0$, let
\begin{equation}
x=\sum_{i=1}^s{(b_i\otimes v_i)} \nonumber
\end{equation}
be such that $\max_i |v_i| \leq \|x\|+\varepsilon$. Denoting $\beta_{ij}$ in $B$ the coefficients of $b_i$ corresponding to $t_j$, we have
\begin{equation}
x=\sum_{j\geq 0}{t_j\otimes \left(\sum_{i=1}^s{\beta_{ij}}v_i\right)}. \nonumber
\end{equation}
In particular, $\|x\|_t=\max_j \left \vert \sum_i{\beta_{ij}}v_i \right \vert \leq \max_i |v_i|\leq \|x\|+\varepsilon$. Since this is true for all $\varepsilon>0$, $\|x\|_t\leq \|x\|$. This proves the first part.\\
Part \ref{item:t2}: By the first part, the assignment $\sum_n t_n\otimes v_n\mapsto (v_n)_{n\geq 0}$ is an isometry between $B\otimes V$ and the space of finite sequences with values in $V$ equipped with the maximum norm. As $V$ is already complete, the completion of this space consists of infinite sequences with values in $V$ which converge to zero. Hence, any element $f$ in $B\hat{\otimes} V$ is of the given form for a unique sequence of elements $(v_n(f))_{n\geq 0}$ in $V$ converging to zero.
\end{proof}

From this explicit description of the completion, one easily sees the following.

\begin{Proposition}\label{prop-completion-exact}
The functor $B\hat{\otimes}-$ from the category of complete normed $\bF$-vector spaces  together with continuous $\bF$-homomorphisms to the category of topological $B$-modules with continuous $B$-homomorphisms is faithful and exact.
\end{Proposition}

\begin{Definition}\label{def:tate-algebra}
We define $\bT$ to be $A\hat{\otimes}\CI$ and call it the \textit{general Tate algebra}. It is an \textit{\'espace de Banach $p$-adique} over $\CI$ in the sense of Serre which satisfies Serre's condition $(N)$ (see \cite{serre}). 
By Proposition \ref{proptate}, $\|x^{(1)}\|=\|x\|^q$ for all $x$ in $A\otimes \CI$. The Frobenius twist is thus continuous, and hence extends to a continuous automorphism of $\bT$. We again denote by $f^{(1)}$ or by $\tau(f)$ the image of $f$ in $\bT$ through this automorphism and we still have $\|f^{(1)}\|=\|f\|^q$. 
\end{Definition}

The norm $\|\cdot\|$ on $\bT$ is multiplicative and the proof is similar to the $A=\bF[t]$ case (see \cite[Sect.~2.2]{bosch}). This implies in particular that $\bT$ is an integral domain. 

\begin{Remark}
The general Tate algebra $\bT$ also has a rigid geometric flavour, i.e. it can be described as a ring of sections on the "unit disc" of the rigid analytic space $(C\times_{\bF}\Spec \CI)^{\text{rig}}$. This equivalent construction is discussed in the beginning of \cite[Sect.~1]{boeckle}.
\end{Remark}

\begin{Example}
In the case where $A=\bF[t]$ is a polynomial ring, by Proposition \ref{proptate}, $\bT$ is canonically isomorphic as a $\CI$-algebra to the usual one-dimensional Tate algebra over $\CI$,
%consisting in all power series
\begin{equation}
\CI\cs{t}= \left\{ \sum_{i=0}^{\infty}{c_it^i}\in \CI\ps{t}\,\, \middle|\,\, c_i\in \CI, \, \lim_{i\to \infty} |c_i|=0 \right\}. \nonumber
\end{equation}
\end{Example}

\bigskip

Let now $(E,\phi)$ be an Anderson $A$-module of dimension $d$.
If we choose an isomorphism $E\cong \bG_{a,L}^d$ of $\bF$-vector space schemes, the maximum norm on $\CI^d$ induces a norm on $E(\CI)$. Although, different choices of isomorphisms induce different (non-equivalent) norms, the set of zero sequences is always the same\footnote{The change from one coordinate system to another is given by an element in the group $\GL_d(L\{\tau\})$, where $L\{\tau\}$ denotes the twisted polynomial ring by the $q$-Frobenius $\tau$ on $L$. When $d>1$, then $\GL_d(L\{\tau\})\neq \GL_d(L)$, and
for $g\in \GL_d(L\{\tau\})\setminus \GL_d(L)$, the corresponding $\bF$-endomorphism of $\CI^d$ is continuous, but not Lipschitz continuous.
Therefore, the two norms built in this way are non-equivalent.}. 
Hence, by the explicit description in Prop.~\ref{proptate}, the completion $A\hat{\otimes}E(\bC_{\infty})$ is independent of the chosen isomorphism. As $A\otimes E(\CI)$ is an $A\otimes A$-module -- the left $A$ acting via multiplication on the $A$-part, and the right $A$ acting via $\phi$ on $E(\CI)$ --, and both actions are continuous, also its completion $A\hat{\otimes}E(\bC_{\infty})$ is an $A\otimes A$-module by extending the actions continuously.
By abuse of notation, we denote by $E(\bT)$ this $A\otimes A$-module. 
Similarly, we define the completion $A\hat{\otimes}\Lie_E(\CI)$ and its $A\otimes A$-module structure, and abbreviate $A\hat{\otimes}\Lie_E(\CI)$ by $\Lie_E(\bT)$.

The exponential map $\exp_E$ is continuous with respect to any norm on $E(\CI)$ given by an isomorphism
$E\cong \bG_{a,L}^d$, and the norm on $\Lie_E(\CI)$ given by the induced isomorphism
$\Lie_E(\CI)\cong \CI^d$. Hence by Prop.~\ref{prop-completion-exact}, we obtain an induced map
$\id\hat{\otimes}\exp_E:\Lie_E(\bT)\to E(\bT)$ as the $A$-linear continuous extension of the exponential map.
From this description, it is clear that $\id\hat{\otimes}\exp_E$ is even a homomorphism of $A\otimes A$-modules.

%SpecialFunctions definition, and comparaison isomorphisms needed for TGTS
%\input{Text/SpecialFunctions}
\section{Special functions}\label{sec:special-functions}

Throughout the whole section, let $E$ be an Anderson-$A$-module over $L$ of dimension $d$. We recall from the last section that $E(\bT):=A\hat{\otimes}E(\CI)$ is an $A\otimes A$-module.

\begin{Definition}
A \textit{special function for} $E$ is an element $\omega\in E(\bT)$ on which the two $A$-actions coincide, i.e. 
\begin{equation}
(a\otimes 1)\omega=(1\otimes a)\omega \quad \forall a\in A.\nonumber
\end{equation}
We denote by $\fsf(E)$ the $A$-module of special functions for $E$. 
\end{Definition}
\begin{Remark}
In \cite{maurischat}, $\fsf(E)$ is denoted $H_E$ for the $A=\bF[t]$ case. The space of special functions is meant to generalize Anderson and Thakur's special function. 
For Drinfeld-Hayes modules over general $A$, special functions were defined by Angl\`es, Ngo Dac and Tavares Ribeiro in \cite[Sect.~3]{tuan}. Take care that their definition differs from ours, as the functions are allowed to lie in the field of fractions of $\bT$. 
When, moreover, $A$ is the coefficient ring of an elliptic curve, they have also been studied by Green and Papanikolas in \cite{green}.
\end{Remark}

\subsection{Relation to the $A$-motive} \label{relation-to-motive}
In this section, we prove a connection between the $A$-module of special functions and the $\tau$-equivariant \textit{Tate-dual} of the motive (Theorem \ref{funceq}). In the case where $E$ is abelian and uniformizable, this connection restricts to an isomorphism of $A$-modules from $\fsf(E)$ to the dual of the Betti realization of $\mot(E)$. 

As we fixed $E$, we abbreviate $\mot:=\mot(E)$.
Let $\mot_{\bT}$ be the $\bT$-module $\mot\otimes_{A\otimes L}\bT$. The action of $\tau_{\mot}$ is extended to $\mot_{\bT}$ by $\tau_{\mot}(m\otimes c):=\tau_{\mot}(m)\otimes c^{(1)}$. We denote by $\Hom_{\bT}^\tau(\mot_{\bT},\bT)$ the $A$-module of $\tau$-equivariant $\bT$-linear morphism, i.e.~of those $\bT$-linear homomorphisms $f:\mot_{\bT}\to \bT$ satisfying $f(x)^{(1)}=f(\tau_{\mot}(x))$ for all $x\in \mot_{\bT}$.
When $E$ is abelian and uniformizable, we consider the \textit{motivic Betti realization} $H_B(\mot)$ of $\mot$ (or $E$) to be the $A$-module:
\begin{equation}
\left\{\sum_{i=1}^s{(m_i\otimes c_i)}\in \mot_{\bT}~\vert~\sum_{i=1}^s{(m_i\otimes c_i)}=\sum_{i=1}^s{\left(\tau_{\mot} (m_i)\otimes c_i^{(1)}\right)}\right\}. \nonumber
\end{equation}
Then $\mot$ is \textit{rigid analytically trival}, i.e. the natural map of $\bT$-modules $H_B(\mot)\otimes_A\bT\rightarrow \mot_{\bT}$ is an isomorphism (see \cite[Thm.~5.28]{hartl}). 
\begin{Theorem}\label{funceq}
There is a natural $A$-linear isomorphism 
\begin{equation}\label{eq:first-iso}
\fsf(E)\longrightarrow \Hom_{\bT}^\tau(\mot_{\bT},\bT).
\end{equation}
If $E$ is abelian and uniformizable, this induces an isomorphism of $A$-modules 
\[ \fsf(E)\longrightarrow \Hom_{A}(H_B(\mot),A). \]
\end{Theorem}

\begin{proof}%{(of Theorem \ref{funceq})}
The idea of the proof for the first part is the same as in \cite[Thm.~3.9]{maurischat}. The natural homomorphism of $\bF$-vector spaces
\begin{equation}\label{eq:iso-E-to-M-dual}
     E(\CI) \longrightarrow \Hom_{L}^{\tau}(\mot, \CI), \quad e\mapsto
\left\{ \mu_e: m\mapsto m(e) \right\}
\end{equation}
is an isomorphism, since after a choice of a $\bF$-vector space scheme morphism $\kappa:E\stackrel{\sim}{=}\bG_{a,L}^d$ over $L$ -- which induces a coordinate system $E(\CI)\cong \CI^d$ -- the latter is isomorphic to the bidual vector space 
$$E(\CI)^{\vee\vee}
=\Hom_{\CI}(\Hom_{\CI}(E(\CI),\CI),\CI)$$ of $E(\CI)$.
The homomorphism \eqref{eq:iso-E-to-M-dual} is even compatible with the $A$-action on $E$ via $\phi$ and the $A$-action on 
$\mu\in \Hom_{L}^{\tau}(\mot, \CI)$ via $(a\cdot \mu)(m)=
\mu(m\circ \phi_a)$ for all $m\in \mot$, $a\in A$.
By tensoring with $A$, we obtain an isomorphims of $A\otimes A$-modules
\[ A\otimes E(\CI) \longrightarrow A\otimes \Hom_{L}^{\tau}(\mot, \CI)
=\Hom_{L}^{\tau}(\mot, A\otimes \CI) \]
Furthermore, a sequence $(e_n)_{n\geq 0}$ of elements in $E(\CI)$ tends to zero if and only if for every $m\in \mot$ the sequence $(m(e_n))_{n\geq 0}$ tends to zero (which is easily seen after a choice of a coordinate system). Hence by taking completions, we obtain an isomorphism of $A\otimes A$-modules
\[ E(\bT) \longrightarrow \Hom_{L}^{\tau}(\mot, \bT). \]
Finally, the image of $\fsf(E)\subset E(\bT)$ consists exactly of those homomorphisms $\mu:\mot\to \bT$ for which $(a\cdot \mu)(m)=(a\otimes 1)\cdot \mu(m)$ for all $
m\in \mot$, $a\in A$. As $(a\cdot \mu)(m)=\mu(m\circ \phi_a)$, these are exactly those homomorphisms which are also $A$-linear, i.e.~the $A\otimes L$-linear ones. By scalar extension of these homomorphisms, we obtain the desired isomorphism
\[  \fsf(E)\longrightarrow \Hom_{A\otimes L}^\tau(\mot,\bT) \cong
\Hom_{\bT}^\tau(\mot_{\bT},\bT). \]
If $E$ is abelian and uniformizable, $\mot$ is rigid analytically trivial. Therefore, 
we have an additional chain of isomorphisms:
\begin{eqnarray*}
\Hom_{\bT}^\tau(\mot_{\bT},\bT) &\stackrel{\cong}{\longrightarrow} & \Hom_{\bT}^\tau(H_B(\mot)\otimes_A\bT ,\bT) \\
&\stackrel{\cong}{\longleftarrow} & \Hom_A^\tau( H_B(\mot),\bT) \\
&\stackrel{\cong}{\longrightarrow} & \Hom_{A}(H_B(\mot),A).
\end{eqnarray*}
Here, the second isomorphism is given by $\bT$-linear extension of homomorphisms, and 
the last isomorphism comes from the fact that $\tau$ acts trivially on $H_B(\mot)$, and hence, the image of a $\tau$-equivariant homomorphism lies in the $\tau$-invariants 
$\{ x\in \bT\mid x^{(1)}=x\}=A$.
\end{proof}

\begin{Remark}\label{rem:explicit-systems-of-equations}
The transition from the isomorphism \eqref{eq:first-iso} to descriptions as in Equation \eqref{eq:drinfeld-solution-tau-solution} is obtained for abelian $A$-modules as follows.
On one hand, one chooses a coordinate system for $E$, i.e.~an isomorphism $\kappa=(\kappa_1,\ldots, \kappa_d):E\to \bG_a^d$, and writes the different $A$-actions as actions on $\bG_a^d(\bT)=\bT^d$. On the other hand, as $\mot$ might not be free over $A\otimes L$, we denote by $Q$ the field of fractions of $A\otimes L$, and choose a basis $m_1,\ldots, m_r$ of $\mot\otimes_{A\otimes L} Q$ as $Q$-vector space (with $m_j\in \mot$), as well as a corresponding dual basis of \[ \Hom_{A\otimes L}^\tau(\mot,\bT)\otimes_{A\otimes L} Q=\Hom_{Q\bT}(\mot_{Q\bT},Q\bT) \]
where $Q\bT$ is the compositum of $Q$ and $\bT$ in the field of fractions of $\bT$.
Then the condition for a homomorphism in $\Hom_{Q\bT}(\mot_{Q\bT},Q\bT)$ to be $\tau$-equivariant is expressed as a $\tau$-difference equation in the coordinates with respect to that basis.

As the coordinate functions $\kappa_1,\ldots, \kappa_d$ for $E$ can be seen as elements in $\mot$ and even provide a $\CI\{\tau\}$-basis of $\mot$, the transition from the solutions of the $A$-action equation to the solutions of the $\tau$-difference equation and back is easily obtained by expressing the $\kappa_i$'s as $Q$-linear combinations of the $m_j$'s, or expressing the $m_j$'s as a $\CI\{\tau\}$-linear combination of the $\kappa_i$'s (cf.~\cite[Sect.~5]{maurischat} for more details in the $\bF[t]$-case).

In calculations in special situations, this isomorphism has been used implicitly,
e.g.~in \cite[\S 2.5]{andersonthakur}, \cite[Sect.~3.2]{tuan}, or \cite[Sect.~5]{green2}.
The situation for Drinfeld-Hayes modules will be explained in more details in Proposition~\ref{mem}.
\end{Remark}

\subsection{A characterization of the relative different}\label{characterisation-of-different}
Let $u$ be an element in $A$ such that the field extension $\bF(u)\subset K$ is finite and separable. In the next section, we will obtain an $A$-module isomorphism from $\fd_{A/\bF[u]}\otimes_A \Lambda_E$ to $\fsf(E)$ where $\fd_{A/\bF[u]}\subseteq A$ is the relative different of the ring extension $\bF[u]\subset A$. To explain the appearance of this ideal, we need a characterization that we now describe (Proposition \ref{different}). As the characterization holds in great generality, we switch only in this subsection to more general notation.

\medskip

Let $\cO$ be a Dedekind domain with fraction field $F$, and let $K$ be a finite separable extension of $F$. Also, we let $A$ be the integral closure of $\cO$ in $K$. To be consistent with \cite[Sect.~III.2]{neukirch}, we shall assume that the residue field extensions of $\cO\subset A$ are separable. We recall that the relative different $\fd_{A/\cO}$ is the ideal of $A$ given by the inverse of the fractional ideal $\{x\in K~|~\forall y\in A:~\tr_{K/F}(xy)\in \cO\}$. However, we will need two other characterizations of this ideal. One characterization is
\[ \fd_{A/\cO} = \{ x\in A \mid \forall y \in A : x\cdot \textrm{d}y=0\in \Omega_{A/\cO}^1\ \}, \]
where $\Omega_{A/\cO}^1$ is the $A$-module of K\"ahler differentials (see \cite[Prop.~2.7]{neukirch}). If $I$ is the kernel of the multiplication map $A\otimes_{\cO} A\xrightarrow{m} A$, 
one has $\Omega_{A/\cO}^1=I/I^2$ and the universal differential $\textrm{d}:A\to \Omega_{A/\cO}^1$ corresponds to $y\mapsto (y\otimes 1-1\otimes y) \mod{I^2}$. The second characterization of the different ideal (\cite[Thm.~III.2.5]{neukirch}) is that $\fd_{A/\cO}$ is generated by all different element $f'(s)$ where $s\in A$ generates the field extension $K/F$, and $f(X)\in \cO[X]$ is the minimal polynomial of $s$ over~$\cO$.

We define
\[ \fD_{A/\cO}:=\{  x\in A\otimes_{\cO} A \mid \forall a\in A: (a\otimes 1)x=(1\otimes a)x \}. \]
We shall prove:
\begin{Proposition}\label{different}
The multiplication map $m:A\otimes_{\cO} A\to A$ restricts to an isomorphism of $A$-modules $\fD_{A/\cO}\cong \fd_{A/\cO}$.
\end{Proposition}

\begin{proof}
We proceed in three steps. \\
1. \textbf{The restriction of the multiplication map $m$ to $\fD_{A/\cO}$ is injective.} \\
The kernel $I$ of $m$ is generated by the elements $(1\otimes a-a\otimes 1)$ for $a\in A$. If $x\in \fD_{A/\cO}$ is in the kernel, then $x\cdot x=0$ by definition of $\fD_{A/\cO}$. However, since $F\subset K$ is separable, the ring $A\otimes_{\cO} A$ does not have nilpotent elements (this is one characterization of separability, see \cite[beginning of Sect.~26]{matsumura}). Hence, $x=0$.
\\
2. \textbf{The image of $\fD_{A/\cO}$ via $m$ is in $\fd_{A/\cO}$.} \\
%Let $I=\ker(m)$ denotes the augmentation ideal. 
The following diagram commutes, for all $y$ in $A$, \\
\centerline{
\xymatrix@C+5mm{
 A\otimes_{\cO} A \ar@{->>}[d]^{m} \ar[r]^(.55){y\otimes 1-1\otimes y} & I \ar[d]^(.45){\textrm{mod} I^2}  \\
   A\ar[r]^(.35){\textrm{d}y} &  \Omega_{A/\cO}^1\cong I/I^2 %:=
}}
thanks to the following calculation ($a$ and $b$ are in $A$):
\begin{eqnarray*}
 ab\cdot \textrm{d}y &=& (ab\otimes 1)(y\otimes 1-1\otimes y) +I^2
 \\ &=& (a\otimes b)(y\otimes 1-1\otimes y) +a(b\otimes 1-1\otimes b)(y\otimes 1-1\otimes y) +I^2 \\ &=&
 (a\otimes b)(y\otimes 1-1\otimes y) +I^2
 \end{eqnarray*}
 By the first characterization of the different, the image of $\fD_{A/\cO}$ via $m$ is in $\fd_{A/\cO}$.
\\
3. \textbf{The restriction $m:\fD_{A/\cO}\to \fd_{A/\cO}$ is surjective.} \\
Let $s$ in $A$ be such that $K=F(s)$. Let $f(X)=X^l+f_{l-1}X^{l-1}+...+f_0$ be the minimal polynomial of $s$ in $\cO[X]$. By the second characterization of the different, it suffices to show that $f'(s)$ is in the image of $\fD_{A/\cO}$ via $m$. The element
\begin{equation}
\sum_{j=0}^{l-1} s^j\otimes  \left( \sum_{k=0}^{l-1-j} f_{l-k}s^{l-1-j-k} \right) \nonumber
\end{equation}
of $A\otimes_{\cO}A$ is mapped to $f'(s)$ under $m$. To conclude, it is enough to prove that it belongs to $\fD_{A/\cO}$. This follows after a simple calculation using the characterization of $\fD_{A/\cO}$ in Lemma \ref{only-consider-s}, and the description of this element in the following remark.
\end{proof}

\begin{Remark}\label{rem:preimage-of-f'(s)}
The element $\sum_{j=0}^{l-1} s^j\otimes 
\left( \sum_{k=0}^{l-1-j} f_{l-k}s^{l-1-j-k} \right)$ can be described much nicer, namely as the evaluation of $\frac{f(X)-f(Y)}{X-Y}\in \cO[X,Y]$ at $X=1\otimes s\in A\otimes_{\cO} A$ and $Y=s\otimes 1\in A\otimes_{\cO} A$.
\end{Remark}

\begin{Lemma}\label{only-consider-s}
Let $s\in A$ be a generator for the field extension $K/F$ 
%-- which exists by the primitive element theorem and the fact that $K=A\cdot F$ -- 
and let $N$ be a torsion-free $A$-module. Then,
\begin{equation}
\{x\in A\otimes_{\cO}N~|~\forall a\in A:~(a\otimes 1)x=(1\otimes a)x\}=\{x\in A\otimes_{\cO}N~|~(s\otimes 1)x=(1\otimes s)x\}. \nonumber
\end{equation}
In particular, an element $x\in A\otimes_{\cO} A$ lies in $\fD_{A/\cO}$ if and only if
$(1\otimes s-s\otimes 1)x=0$.
\end{Lemma}
\begin{proof}
One inclusion is clear, so we prove the opposite one. For $x\in A\otimes_{\cO}N$, we first notice that the given condition implies 
$(b\otimes 1)x=(1\otimes b)x$ for all $b\in \cO[s]$. For arbitrary $a\in A$,
we write $a=b/c$ with $b$ and $c$ in $\cO[s]\subset A$. Then,
\begin{align*}
(c\otimes 1)(a\otimes 1)x &= (b\otimes 1)x =(1\otimes b)x =(1\otimes ac) x \\
&= (1\otimes a)(1\otimes c)x = (1\otimes a)(c\otimes 1)x=(c\otimes 1)(1\otimes a)x.
\end{align*}
The result follows by dividing out $(c\otimes 1)$ as $N$ is torsion-free.
\end{proof}

We end this section by a Lemma that will be needed afterwards.

\begin{Lemma}\label{different-module}
Let $N$ be a torsion-free $A$-module. The multiplication map $A\otimes_{\cO}N \to N$ induces an isomorphism of $A$-modules
\begin{equation}
\left\{x\in A\otimes_{\cO}N~\vert~\forall a\in A:~(a\otimes 1)x=(1\otimes a)x\right\}\cong \fd_{A/\cO}\cdot N. \label{eq:DN-dN}
\end{equation}
\end{Lemma}

\begin{proof} 
Let $s\in A$ be a generator for the field extension $K/F$. By Lemma \ref{only-consider-s}, $\fD_{A/\cO}$ is the kernel of the map
$\varepsilon:A\otimes_{\cO} A\to A\otimes_{\cO} A$ given by multiplication with $1\otimes s-s\otimes 1$. Hence,
\[ 0\to \fD_{A/\cO}\longrightarrow A\otimes_{\cO} A \xrightarrow{\varepsilon} A\otimes_{\cO} A \]
is exact. As $N$ is torsion-free and $A$ is a Dedekind domain, $N$ is a flat $A$-module. Therefore, tensoring with $N$ over $A$ is exact, and 
hence, $\fD_{A/\cO}\otimes_A N=\ker(\varepsilon)\otimes_A N\cong \ker(\varepsilon\otimes \id_N)$. By Lemma \ref{only-consider-s}, $\ker(\varepsilon\otimes \id_N)$ equals the left hand side of the isomorphism \eqref{eq:DN-dN}. On the other hand, using Proposition \ref{different}, we see that $\fD_{A/\cO}\otimes_A N$ is also isomorphic to $\fd_{A/\cO}\otimes_A N\cong \fd_{A/\cO}\cdot N$.
\end{proof}

\subsection{Relation to the period lattice}\label{subsec:relation-to-lattice}
\begin{Definition}
An element $u$ in $A$ for which the field extension $K/\bF(u)$ is finite and separable will be called \textit{separating}.\footnote{Namely, in this case the set $\{u\}$ is a separating transcendence basis of $K$.} $\bF[u]\subset A$ is then a finite separable integral extension of $\bF$-algebras which meets the conditions of the extension $A/\cO$ in the previous subsection. 
\end{Definition}

\begin{Example}
A uniformizer $u_{\fp}$ of any non-zero prime ideal $\fp$ in $A$  is a separating element. Indeed, the closed point $\fp$ of $(C,\cO_{C})$ is unramified under the map $C\rightarrow \bP^1_{\bF}$ corresponding to the inclusion $\bF(u_{\fp})\subset K$. As the inseparability degree divides all ramification indices, the extension $\bF(u_{\fp})\subset K$ is separable.
\end{Example}

Let $u$ be a separating element in $A$. In this section, we establish a natural isomorphism of $A$-modules between $\fd_{A/\bF[u]}\cdot \Lambda_E$ and $\fsf(E)$ (see Theorem \ref{lattice}) -- which, however, depends on $u$. The surprising consequence of this fact is that the period lattice and the module of special functions might not be isomorphic.

\begin{Theorem}\label{lattice}
Let $u$ be a separating element in $A$. The sequence of $A$-module homomorphisms
\[ \Lambda_E \longleftarrow A\otimes_{\bF[u]}\Lambda_E \stackrel{\tilde{\delta}_u}{\longrightarrow} \left\{\omega\in E(\bT) \mid (u\otimes 1)\omega=(1\otimes u)\omega \right\},\]
where the first arrow is the multiplication map and the second one takes an element $x$ of $A\otimes_{\bF[u]}\Lambda_E$ to $\tilde{\delta}_u(x)=(\id\hat{\otimes}\exp_E)\left((1 \otimes \partial u-u\otimes 1)^{-1}x\right)$, induces an isomorphism of $A$-modules 
\[  \delta_u:\fd_{A/\bF[u]}\cdot \Lambda_E \stackrel{\cong}{\longrightarrow} \fsf(E). \]
\end{Theorem}

\begin{proof}
We first explain why $\tilde{\delta}_u$ is well-defined, and that it even is an isomorphism. This is merely the same as in \cite[Sect.~3.2]{maurischat}.

The exact sequence of $A$-modules (via $\partial \phi$ and $\phi$ resp.)
\[0 \longrightarrow \Lambda_E \longrightarrow \Lie_E(\CI) \xrightarrow{\exp_E} E(\CI)\]
induces a sequence of $A\otimes A$-modules
\[ 0\longrightarrow A\otimes \Lambda_E \longrightarrow \Lie_E(\bT) \xrightarrow{\id\hat{\otimes}\exp_E} E(\bT), \]
by tensoring with $A$ and taking completions. Here, we take into account that $\Lambda_E$ is discrete, and hence $A\hat{\otimes}\Lambda_E=A\otimes \Lambda_E$. By Proposition \ref{prop-completion-exact} this sequence is exact again.

Now, let $u$ be the separating element in $A$ from the statement of the theorem.
The map $\Lie_E(\bT)\to \Lie_E(\bT),x\mapsto (1\otimes \partial u-u\otimes 1)(x)$ is $\bT$-linear and its determinant is a power of $(1\otimes \ell(u)-u\otimes 1)\in \bT$. Since $|\ell(u)|>1$, the latter is even invertible in $\bT$ with inverse
\[ (1\otimes \ell(u)-u\otimes 1)^{-1}=\sum_{i\geq 0} u^i\otimes \ell(u)^{-i-1}\in \bT, \]
and hence the homomorphism $1\otimes \partial u-u\otimes 1$ is an isomorphism. Therefore, as in \cite[Sect.~3.2]{maurischat}, one obtains a commuting diagram of $A\otimes A$-modules with exact rows:\\ 
\centerline{
\xymatrix@C+10mm{
 & 0 \ar[r] \ar@{-->}[d]  &  A\otimes\Lambda_E  \ar[d]  \ar[r]^{1\otimes \partial u-u\otimes 1} &  A\otimes\Lambda_E  \ar[d] 
%\ar@{-->} '[r]'[dddlll] [dddll] 
    \ar@{-->}'[r] `r[d]`[d]+/d 3ex/ `^d[lll]+/l 3ex/ `[dddll]      [dddll]
 & \\
 0 \ar[r] & 0\ar[d] \ar[r] &  \Lie_E(\bT)  \ar[d]^{\id\hat{\otimes}\exp_E} \ar[r]^{1\otimes \partial u-u\otimes 1} &  
 \Lie_E(\bT)\ar[d]^{\id\hat{\otimes}\exp_E}\ar[r]  & 0 \\
 0 \ar[r] & H \ar@{=}[d] \ar[r]  & E(\bT) \ar[r]^{1\otimes u-u\otimes 1} & E(\bT) &\\
 & H &&&
}}
where $H$ is defined to be 
\[ H:= \ker(1\otimes u-u\otimes 1)=\left\{\omega\in E(\bT) \mid (u\otimes 1)\omega=(1\otimes u)\omega \right\}.\] 
The snake lemma then induces the dashed arrow, and hence an injective $A\otimes A$-homomorphism
$\tilde{\delta}_u:A\otimes_{\bF[u]} \Lambda_E \to H$. By diagram chasing, we see that for $x\in A\otimes \Lambda_E$, $\tilde{\delta}_u(x)$ is given by
\begin{equation}
\tilde{\delta}_u(x)=(\id\hat{\otimes}\exp_E)\left((1 \otimes \partial u-u\otimes 1)^{-1}x\right). \nonumber
\end{equation}
The same proof as in \cite[Thm.~3.6]{maurischat} shows that $\tilde{\delta}_u$ is also surjective. The main ideas are as follows: First choose $t_1,\ldots, t_k\in A$ such that their residues modulo $u$ generate the finite dimensional $\bF$-vector space $A/(u)$, and extend this to an $\bF$-basis of $A$ by taking all $u$-power multiples of these elements. After the choice of an $\bF$-vector space scheme isomorphism $\kappa$ over $L$ from $E$ to $\bG_{a,L}^d$, we define a norm $|\cdot|$ on $E(\CI)\cong \CI^d$ and on $\Lie_E(\CI)\cong \CI^d$. By Proposition \ref{proptate}, we write accordingly an arbitrary $h\in E(\bT)$ as $h=\sum_{n\geq 0} \sum_{j=1}^k t_ju^n\otimes e_{j,n}$ where $e_{j,n}\in E(\CI)$. If further $h\in H$, the $e_{j,0}$ are $u$-torsion elements, and $u(e_{j,n})=e_{j,n-1}$ for all $n\geq 1$, $j\in\{1,\ldots , k\}$. Furthermore, as $(e_{j,n})_n$ tends to $0$ with respect to $|\cdot|$, \cite[Lem.~5.3]{hartl} shows that for large enough $n$, there exists a unique $L_{j,n}\in \Lie_E(\CI)$ such that $\exp_E(L_{j,n})=e_{j,n}$ and $|L_{j,n}|=|e_{j,n}|$. It allows one to define 
\[  \lambda:=\sum_{j=1}^k t_j\otimes \partial u^{n+1} L_{j,n}. \]
A calculation shows that a) $\lambda$ is independent of the chosen $n$, as long as $n$ is large enough, that b) $\lambda$ is indeed in $A\otimes_{\bF[u]} \Lambda_E$, and that c) 
$\tilde{\delta}_u(\lambda)=h$.

The isomorphism $\tilde{\delta}_u$ then restricts to an isomorphism of $(A\otimes A)$-modules\footnote{Actually, both $A$-actions coincide by definition, and we can just consider them as $A$-modules via either of the two actions.}
\[ \fD_{A/\bF[u]}\otimes_A \Lambda_E= \left\{x\in A\otimes_{\bF[u]} \Lambda_E~ \vert ~ \forall a\in A:~ (a\otimes 1)x=(1\otimes a)x \right\} \longrightarrow \fsf(E). \]
By Lemma \ref{different-module}, the left-hand-side is isomorphic to $\fd_{A/\bF[u]}\cdot \Lambda_E$ via the multiplication map $A\otimes_{\bF[u]} \Lambda_E \to \Lambda_E$. This leads to the desired isomorphism
 \begin{equation} \delta_u:\fd_{A/\bF[u]}\cdot \Lambda_E\cong \fD_{A/\bF[u]}\otimes_A \Lambda_E \stackrel{\tilde{\delta}_u}{\longrightarrow} \fsf(E).\nonumber  \qedhere \end{equation}
\end{proof}

We immediately obtain the following corollary.
\begin{Corollary}\label{cor:iso-lattice-sf}
If there exists a separating element $u$ in $A$ for which $A=\bF[u,s]$ for some $s$ in $A$, then we have an isomorphism of $A$-modules $\Lambda_E \cong \fsf(E)$.
\end{Corollary}
\begin{proof}
If $A=\bF[u,s]$, $\fd_{A/\bF[u]}$ is principal and generated by $f'(s)$ where $f(X)$ is the minimal polynomial of $s$ over $\bF[u]$ (see \cite[Prop.~III.2.4]{neukirch}). Therefore, it suffices to compose the isomorphism $\delta_u$ of Theorem \ref{lattice} with multiplication by $f'(s)$.
\end{proof}

The isomorphism given in Theorem \ref{lattice} obviously depends on $u$, and it seems that all the induced isomorphisms $\Lambda_E\cong \fsf(E)$ in Corollary \ref{cor:iso-lattice-sf} even depend on $u$ and $s$. The following proposition, however, shows that in the case where $A=\bF[t]$ is the polynomial ring, all the isomorphisms coincide.

\begin{Proposition}\label{prop:delta-always-the-same}
Assume that $A=\bF[t]$ is a polynomial ring, and let $u(t)\in \bF[t]$ be separating. Then the composition
\[   \Lambda_E \stackrel{u'(t)\cdot }{\longrightarrow} \fd_{A/\bF[u]}\cdot \Lambda_E \stackrel{\delta_u}{\longrightarrow} \fsf(E) \]
is independent of the chosen element $u(t)$. In particular, it equals the isomorphism
\[ \delta_t:\Lambda_E \to \fsf(E) \]
sending a period $\lambda\in \Lambda_E$ to its Anderson generating function
\[   (\id\hat{\otimes}\exp_E)\left( (1\otimes \partial t - t\otimes 1)^{-1} \lambda\right)
= \sum_{n=0}^\infty t^n\otimes \exp_E\left((\partial t)^{-n-1}\lambda \right). \]
\end{Proposition}

\begin{proof}
As mentioned before, the different ideal $\fd_{\bF[t]/\bF[u]}$ is generated by $u'(t)$, and the preimage of $u'(t)$  in $\fD_{\bF[t]/\bF[u]}\subseteq \bF[t] \otimes_{\bF[u]} \bF[t]$ 
is $U:=\frac{u(X)-u(Y)}{X-Y}|_{X=1\otimes t, Y=t\otimes 1}$ by Remark \ref{rem:preimage-of-f'(s)}.
Therefore, by the definition of $\delta_u$, we explicitly obtain for any $\lambda\in \Lambda_E$
\begin{eqnarray*} \delta_u(u'(t)\cdot \lambda) &=& U\cdot (\id\hat{\otimes}\exp_E)\left( (1\otimes \partial u - u\otimes 1)^{-1} \lambda\right) \\
&=& (\id\hat{\otimes}\exp_E)\left( (1\otimes \partial u - u\otimes 1)\cdot (1\otimes \partial t - t\otimes 1)^{-1}\cdot (1\otimes \partial u - u\otimes 1)^{-1} \lambda\right) \\
&=&  \delta_t(\lambda). \qedhere
\end{eqnarray*}
\end{proof}

\begin{Remark}\label{rmk:different-canonical}
It is well-known that the class of the different ideal $[\fd_{A/\bF[u]}]$ in the class group of $A$ is equal to $-[\Omega_{A/\bF}^1]$. This is an immediate application of the Riemann-Hurwitz theorem to the finite and separable morphism of curves $C\to \bP^1_{\bF}$ given by the choice of $u$ (see \cite[Prop.~IV.2.3]{hartshorne}). The Riemann-Hurwitz theorem does not, however, give the $A$-module isomorphism from $\Omega_{A/\bF}^1$ to $\fd_{A/\bF[u]}^{-1}$. In this remark, we provide an alternative reason why they are isomorphic. By \cite[Prop.~2.7]{neukirch}, we have
\[ \fd_{A/\bF[u]} = \{ x\in A \mid \forall y\in A: x\cdot \textrm{d}y=0\in \Omega_{A/\bF[u]}^1 \}. \]
Besides, as $\bF[u]\subset A$ is separable, one has an exact sequence of $A$-modules:
\begin{equation}
0\longrightarrow \Omega^1_{\bF[u]/\bF}\otimes_{\bF[u]} A\longrightarrow \Omega^1_{A/\bF}\longrightarrow \Omega_{A/\bF[u]}^1\longrightarrow 0. \nonumber
\end{equation}
In particular, we can rewrite $\fd_{A/\bF[u]}$ as:
\[ \fd_{A/\bF[u]} = \{ x\in A \mid \forall y\in A: x\cdot \textrm{d}y\in Adu\}. \]
Thanks to this description, the equality $\fd_{A/\bF[u]}\cdot \Omega_{A/\bF}^1=Adu$ is clear.

This explains why the different ideal is not principal in general. Indeed, we then have
\begin{equation}
\deg \fd_{A/\bF[u]}\equiv -\deg \Omega_{A/\bF}^1\equiv 2-2g \pmod{d_{\infty}} \nonumber
\end{equation}
where $g$ is the genus of $C$. Hence, the order of $[\fd_{A/\bF[u]}]$ is divisible by $d_{\infty}/\operatorname{gcd}(2g-2,d_{\infty})$.
\end{Remark}

In \cite[end of Sect.~3.2]{tuan}, it is asked whether the space of special functions is always free as an $A$-module. The question is raised in rank and dimension $1$, but we are able to answer it in
the negative for any rank $r\geq 1$:
\begin{Corollary}\label{special-function-free}
Let $\Lambda$ be an $A$-lattice in $\CI$ projective of rank $r\geq 1$ and let $\CD$ be the associated Drinfeld $A$-module over $\CI$ (according to \cite[Section~4.3]{goss}). Then, $\fsf(\CD)$ is free as an $A$-module if, and only if, $\det_A(\Lambda)$ and $(\Omega_{A/\bF}^1)^{\otimes r}$ are isomorphic as $A$-modules.
\end{Corollary}
\begin{proof}
As a torsion-free module over the Dedekind domain $A$, $\Lambda$ is isomorphic to $A^{r-1}\oplus \det_A(\Lambda)$. For any separating element $u$, it yields 
\begin{equation}\label{eq:chain}
\fd_{A/\bF[u]}\cdot\Lambda\cong \fd_{A/\bF[u]}^{r-1}\oplus \fd_{A/\bF[u]}\operatorname{det}_A(\Lambda)\cong A^{r-1}\oplus \fd_{A/\bF[u]}^{r}\operatorname{det}_A(\Lambda).
\end{equation}
The construction implicates that $\Lambda$ is the period lattice of $\CD$ so that, by Theorem \ref{lattice}, \eqref{eq:chain} is isomorphic to $\fsf(\CD)$. Hence, the module $\fsf(\CD)$ is free as an $A$-module if, and only if $\fd_{A/\bF[u]}^{r}\det_A(\Lambda)$ is isomorphic to $A$. We conclude by Remark \ref{rmk:different-canonical}.
\end{proof}

\begin{Remark}
Assume that $E$ is an abelian and uniformizable $A$-module. Combining the $A$-module isomorphisms from Theorems \ref{lattice} and \ref{funceq}, namely
\begin{equation}
\fd_{A/\bF[u]}\cdot \Lambda_E\longrightarrow \fsf(E), \quad \fsf(E)\longrightarrow \Hom_A(H_B(\mot),A) \nonumber 
\end{equation}
together with $\Omega_{A/\bF}^1\cong \fd_{A\bF[u]}^{-1}$ as $A$-modules from Remark \ref{rmk:different-canonical}, one gets an $A$-module isomorphism $\Lambda_E\to \Hom_A(H_B(\mot),\Omega_{A/\bF}^1)$ between the period lattice of $E$ and its Betti cohomology. Such a relation was already known \textit{via} a map denoted $\beta_A$ and follows from \cite[Cor.~2.12.1]{anderson} (also described in \cite[Rmk.~5.30]{hartl}). It is therefore natural to ask whether the two approaches match. They are indeed related to each other by the following diagram of $A$-modules which is commutative up to a sign and each arrow is an isomorphism: \\
\centerline{
\xymatrix@C-0.1cm{
\fd_{A/\bF[u]}\Lambda_E \ar[d]_{\delta_u}& \fd_{A/\bF[u]}\otimes_A\Lambda_E \ar[r]^(.35){\id \otimes_A \beta_A} \ar[l] & \fd_{A/\bF[u]}\otimes_A \Hom_A(H_B(\mot),\Omega^1_{A/\bF}) \\
\fsf(E) \ar[r]^(.3){\text{Thm}~\ref{funceq}} & \Hom_A(H_B(\mot),A) & (\fd_{A/\bF[u]}\otimes_A \Omega_{A/\bF}^1)\otimes_A \Hom_A(H_B(\mot),A) \ar[l] \ar[u]
}
}
The commutativity up to a sign is not obvious and relies on a non-trivial calculation.  
\end{Remark}

\subsection{Special functions of Drinfeld-Hayes modules}\label{drinfeld-hayes}

This section aims not only to connect Thakur's shtuka function theory with special functions of Drinfeld-Hayes modules, but also to see how our definition of special functions matches \cite[Def.~3.9]{tuan}.

Given an abelian Anderson $A$-module over $L$ of rank and dimension one, Thakur in \cite{thakur93} was able to define an element $f$ in the fraction field of $A\otimes L$, called the \textit{shtuka function of $E$}, of fundamental importance for $E$. His construction of $f$ relies on the Drinfeld-module/shtuka dictionary. We should present a slightly more general alternative definition which avoids the use of shtukas.

In this section, we assume that $E$ is an abelian Anderson $A$-module of rank and dimension one, namely a \textit{Drinfeld-Hayes module}. Let $\kappa:E\stackrel{\simeq}{\longrightarrow} \bG_{a,L}$ be an $\bF$-vector space schemes isomorphism over $L$. As a non-zero element of the rank one projective $A\otimes L$-module $\mot=\mot(E)$, $\kappa$ is proportional to the element $\Frob_q\circ \kappa$ in $\mot$, where $\Frob_q$ denotes the $q$-Frobenius on $\bG_{a,L}$. This justifies the existence-part of following definition:
\begin{Definition}
We call the \textit{shtuka function of $E$ associated to $\kappa$}, and write it $f_{\kappa}$, the unique element in the fraction field of $A\otimes L$ such that $\Frob_q\circ \kappa=f_\kappa \cdot \kappa$.
\end{Definition} 
The morphism $\kappa$ also provides a family $\{\phi_a^{\kappa}\}_{a\in A}$ of $\bF$-linear endomorphisms of $\bG_{a,L}$, where $\phi_a^{\kappa}:=\kappa^{-1}\circ \phi_a\circ \kappa$. Seen as elements of $L\{\tau\}$, each $\phi_a^{\kappa}$ can be written uniquely as
\begin{equation}\label{eq:tau_expansion}
\phi_a^{\kappa}=(a)_0^{\kappa}+(a)_1^{\kappa}\tau+(a)_2^{\kappa}\tau^2+...+(a)_d^{\kappa}\tau^d\quad (d:=\deg(a),~(a)_i^{\kappa}\in L).
\end{equation} 

A brief description of the sign-normalized condition is called for. Let $t^{-1}$ be a local parameter at a point $\overline{\infty}$ above $\infty$ in $C\times_{\bF} \Spec\CI$. Any $x$ in the function field of $C\times_{\bF} \Spec\CI$ can be written uniquely in a Laurent series in $t^{-1}$:
\begin{equation}
x=\sum_{i\geq n}{c_it^{-i}} \quad (c_i\in \CI,~n\in \bZ) \nonumber
\end{equation}
with the coefficient $c_n$ being nonzero. We define $\sgn(x)$ to be the latter coefficient $c_n$ in $\CI^{\times}$ and call it the \textit{sign of $x$ (with respect to $t$)}. The function $\sgn$ then defines a group morphism from the group of nonzero elements of the function field of $C\times_{\bF} \Spec\CI$ to $\CI^{\times}$. In \cite{thakur93}, it is assumed that the assignment
\begin{equation}
A\setminus \{0\}\longrightarrow L\subset \CI, \quad a\longmapsto (a)_d^{\kappa} \quad (\text{with~}d=\deg(a))
\end{equation}
coincides with $\sgn$. If this is the case, we should say that $\kappa$ is \textit{sgn-normalized (with respect to $t$)}. As we do not make such an assumption here, we will assume that $\kappa$ is arbitrary.

\medskip

Thanks to the isomorphism of $A$-modules $\id\otimes\kappa:E(\bT)\to \bT$, we can see the elements of $\fsf(E) \subset E(\bT)$ as elements of $\bT$ satisfying a family of functional equations:
\begin{equation}\label{eq:explicit_special_function}
\fsf(E)\stackrel{1 \otimes \kappa}{\longrightarrow} \{\omega\in \bT~|~ \forall a\in A : (1\otimes \phi_a^{\kappa})(\omega)=(a\otimes 1)\omega \}. 
\end{equation} 
The above map is an isomorphism of $A$-modules. This seems at first sight in contrast with the definition of special functions as given in \cite[Def.~3.9]{tuan} where they consist of elements $\omega$ in the general Tate algebra satisfying $\omega^{(1)}=f_{\kappa}\omega$. We should see that the two definitions agree:
\begin{Proposition}\label{mem}
For any isomorphism $\kappa:E\stackrel{\simeq}{\longrightarrow}\bG_{a,L}$ of $\bF$-vector space schemes over $L$, we have the following equality of subsets of $\bT$:
\begin{equation}
\{\omega\in \bT~|~ \forall a\in A : (1\otimes \phi_a^{\kappa})(\omega)=(a\otimes 1)\omega \}=\{\omega\in \bT~|~ \omega^{(1)}=f_{\kappa}\omega \}. \nonumber
\end{equation}
\end{Proposition}
Before proving Proposition \ref{mem}, we investigate the shape of the divisor of $f_{\kappa}$ on $C\times_{\bF} \Spec L$. As $\mot$ is a rank one projective $A\otimes L$-module, there exists an ideal $\fv_{\kappa}$ of  $A\otimes L$ such that $(A\otimes L)\cdot \kappa=\fv_{\kappa}\mot$. We let $\fj$ be the maximal ideal of $A\otimes L$ generated by the set
\begin{equation}
\{(a\otimes1)-(1\otimes \ell(a)) \mid a \in A\}. \nonumber
\end{equation}
\begin{Lemma}\label{divisor}
Let $V_{\kappa}$ be the effective divisor on $\Div(C\times_{\bF}\Spec L)$ given by the primary decomposition of $\fv_{\kappa}$. There exists a divisor $\infty_{\kappa}$ of degree $1$ and supported at $\infty$, such that
\begin{equation}
\operatorname{div} f_{\kappa}=\tau^*V_{\kappa}-V_{\kappa}+\fj-\infty_{\kappa}.
\end{equation}
\end{Lemma}
\begin{proof}
We borrow the following exact sequence from \cite[Lem.~3.1.4]{anderson}:
\begin{equation}\label{andersonexact}
0\longrightarrow \tau_{\mot}(\tau^*\mot)\longrightarrow \mot \longrightarrow \Hom_{L}(\Lie_E(L),L)\longrightarrow 0
\end{equation}
where the second arrow is the inclusion of $A\otimes L$-modules, and the third one is the differential at the origin. The exact sequence (\ref{andersonexact}) implies $\fj \mot \subset \tau_{\mot}(\tau^*\mot)\subset \mot$. As $\mot/\fj\mot$ is a one-dimensional $L$-vector space, $\tau_{\mot}(\tau^*\mot)$ equals either $\fj \mot$ or $\mot$. The last option is impossible since composition by the $q$-Frobenius on the motive is not a surjective operation (as $E$ is isomorphic to $\bG_{a,L}$). Hence, we have $\fj \mot=\tau_{\mot}(\tau^*\mot)$. Writing $\mot=\fv^{-1}_{\kappa}\cdot \kappa$ leads to the identity $\fj \fv_{\kappa}^{-1}=(f_{\kappa})(\tau^*\fv_{\kappa})^{-1}$ of fractional ideals. This is equivalent to 
\begin{equation}
(\operatorname{div} f_{\kappa})|_{\Spec A\otimes L}=\tau^*V_{\kappa}-V_{\kappa}+\fj \nonumber
\end{equation}
on the affine part of the curve, where $V_{\kappa}$ is the effective divisor associated to the ideal $\fv_{\kappa}$. Therefore, there exists $\infty_{\kappa}$ above $\infty$ for which
\begin{equation}
\operatorname{div} f_{\kappa}=\tau^*V_{\kappa}-V_{\kappa}+\fj-\infty_{\kappa}. \nonumber
\end{equation}
The ideal $\fj$ has degree one when seen as a divisor on $\Spec (A\otimes L)$ as its residue field is $L$. That the degree of $\infty_{\kappa}$ is $1$ follows from $\deg f_{\kappa}=0$ as a principal divisor on $C\times_{\bF}\Spec L$.
\end{proof}

\begin{proof}[Proof of Proposition \ref{mem}]
The proof of Theorem \ref{funceq} implies that the natural isomorphism\linebreak $\fsf(E)\to \Hom_{\bT}^{\tau}(\mot_{\bT},\bT)$ induces, via $\id\otimes \kappa$, the equality
\begin{equation}
\{\omega\in \bT \mid \forall a\in A: (1\otimes\phi_a^{\kappa})(\omega)=(a\otimes 1)\omega\}=\{\omega\in\fv_{\kappa}\bT \mid \omega^{(1)}=f_{\kappa}\omega\}.\nonumber
\end{equation}
We claim that $\{\omega\in\fv_{\kappa}\bT \mid \omega^{(1)}=f_{\kappa}\omega\}=\{\omega\in \bT\mid \omega^{(1)}=f_{\kappa}\omega\}$. Indeed, if $\omega \in \bT$ satisfies $\omega^{(1)}=f_{\kappa}\omega$, we have that $\omega$ belongs to the intersection $f^{-1}_{\kappa}\bT \cap \bT$. As $V_{\kappa}$ is effective, Lemma \ref{divisor} yields $f_{\kappa}^{-1}\bT \cap \bT=(\fj\tau^*\fv_{\kappa})^{-1}\fv_{\kappa}\bT\cap \bT$ which is a subset of $\fv_{\kappa}\bT$.
\end{proof}
Finally, we discuss two corollaries of Proposition \ref{mem}.
\begin{Corollary}\label{invertible}
Assume that there exists $\omega$ in $\bT^{\times}$ such that $\omega^{(1)}=f_{\kappa}\omega$. Then, the period lattice $\Lambda_E$ is isomorphic to $\Omega_{A/\bF}^1$ as an $A$-module.
\end{Corollary}

\begin{proof}
Let $\omega$ be as in the statement and let $\omega'$ be an element of $\bT$ satisfying $\omega'^{(1)}=f_{\kappa}\omega'$. Then $(\omega'/\omega)^{(1)}=\omega'/\omega$ in $\bT$ so that $\omega'\in A\omega$. Thus $\{\omega\in \bT~ |~\omega^{(1)}=f_{\kappa}\omega\}$ is free of rank $1$ as an $A$-module. We conclude by Proposition \ref{mem}, Equation \eqref{eq:explicit_special_function} and Corollary \ref{special-function-free}.
\end{proof}
We do not know whether the converse of Corollary \ref{invertible} holds, and leave it as an open question:
\begin{Question*}
If the period lattice $\Lambda_E$ is isomorphic to $\Omega_{A/\bF}^1$ as an $A$-module, does there exist $\omega\in \bT^{\times}$ such that $\omega^{(1)}=f_{\kappa}\omega$?
\end{Question*}

\begin{Corollary}\label{start}
For all $a \in A$, we have the following identity in the fraction field of $A\otimes L$:
\begin{equation}
(a\otimes 1)=\sum_{n\geq 0}{(a)_n^{\kappa}f^{(n-1)}_{\kappa}\cdots f^{(1)}_{\kappa}f_{\kappa}}. \nonumber
\end{equation}
\end{Corollary}
\begin{proof}
By Proposition \ref{mem}, there exists a non zero element $\omega$ in the general Tate algebra such that $\omega^{(1)}=f_{\kappa}\omega$ and also, for all $a$ in $A$,
\begin{equation}
(a\otimes 1)\omega=(1\otimes \phi_a^{\kappa})(\omega)=\sum_{n\geq 0}{(a)_n^{\kappa}\omega^{(n)}}=\sum_{n\geq 0}{(a)_n^{\kappa}f^{(n-1)}_{\kappa}\cdots f^{(1)}_{\kappa}f_{\kappa}\omega}. \nonumber
\end{equation}
The corollary follows by dividing out $\omega$.
\end{proof}
\begin{Remark}
The formula of Corollary \ref{start} is an immediate property of the shtuka function as defined by Thakur (see \cite[Eq.~$(\star \star)$]{thakur93}). Note that it already appears in \cite[Eq.~(46)]{andersonrankone}. This formula has been the starting point of \cite{tuan} and \cite{green} to define special functions. For our presentation, we have taken the opposite path.
\end{Remark}

%Definition and straightforward properties of TGTS
%\input{Text/TensorGaussThakurSums}
\section{Gauss-Thakur sums}\label{sec:gauss-thakur-sums}
In this section, we introduce Gauss-Thakur sums attached to an arbitrary Anderson $A$-module $E$ over $L$. Our definition follows closely Thakur's one (see \cite{thakur88}) although there are some natural changes that we shall explain. 

\medskip

From now on, we fix $\fp$ a maximal ideal in $A$, $\Fp$ its residue field and $\degp$ its degree, i.e. the dimension of $\Fp$ over $\bF$. We denote by $E[\fp]\subset E(\CI)$ the $A$-module of $\fp$-torsion points of $E$ over $\CI$, i.e.~the space $E[\fp]=\{ e\in E(\CI)\mid~\forall a\in \fp:~ a e=0\}$.
\begin{Definition}\label{gaussthakurdef}
A \textit{multiplicative character} $\chi$ (of conductor $\fp$) is a group morphism from $(A/\fp)^{\times}$ to $\Fp^{\times}\subset \CI^{\times}$. An \textit{additive character} $\psi$ (for $E$ and $\fp$) is an $A$-module morphism from $A/\fp$ to $E[\fp]\subset E(\CI)$. Given such $\chi$ and $\psi$, we define their \textit{(tensor) Gauss-Thakur sum} as
\begin{equation}
g(\chi,\psi):= -\sum_{x\in (A/\fp)^{\times}}{\chi(x)^{-1}\otimes \psi(x)} \quad \in \Fp\otimes E[\fp]. \nonumber
\end{equation}
A multiplicative character $\chi$ will be lifted to a map $A\to \Fp$ by sending
$a\in A$ to $\chi(a+\fp)$ if $a\notin \fp$, and to $0$ if $a\in \fp$. By abuse of notation, this lift will also be denoted by $\chi$. 

\medskip

Using this lift, we define the space
\begin{equation}
G(E,\chi):=\left\{g\in \Fp\otimes E(\bC_{\infty})~\vert ~\forall \, a\in A:~ (1\otimes a)(g)=(\chi(a)\otimes 1)g\right\}, \nonumber
\end{equation}
and call it the \textit{space of Gauss-Thakur sums}. \end{Definition}
The naming ``space of Gauss-Thakur sums'' will be justified by Lemma \ref{projection-on-gauss-space}\ref{item:c} and Proposition \ref{nonzero} where we show that $G(E,\chi)$ is the $\Fp$-vector space generated by the Gauss-Thakur sums.
\begin{Remark}
The minus sign in the definition of Gauss-Thakur sums already appear in \cite{thakur88}. It serves as the normalization factor $|\Fp|-1=-1\in \bF$.
\end{Remark}

\begin{Remark}\label{multiplicativemap}
As opposed to Thakur's definition, a tensor product replaces the multiplication. Indeed, in general, $E$ is not \textit{equal} to $\bG_{a,L}^d$ as an $\bF$-vector space scheme but only \textit{isomorphic}. As such, it does not carry a canonical $L$-vector space structure and the product $\chi(x)^{-1}\psi(x)$ is not even defined. When $E$ equals $\bG_{a,L}^d$ as an $\bF$-vector space scheme, for instance in Thakur's consideration when $d=1$, or after fixing an isomorphism $E\simeq \bG_{a,L}^d$, one can define the \textit{classical} Gauss-Thakur sum as being the tensorless version of the sum (\ref{thakurgausssum}) in $E(\overline{L})=\overline{L}^d$. In these cases, the reader will be free to consider classical ones via the multiplication map $\Fp\otimes E[\fp]\rightarrow E(\overline{L})$. We will use this tensorless version in Section \ref{L-values} for the sake of a simplified presentation of the results. Note, however, that this morphism is injective if, and only if, the field of definition of the $\fp$-torsion points is linearly independent of $\Fp$.
\end{Remark}

\begin{Example}\label{psi-e}
For every torsion point $e\in E[\fp]$, one obtains a well-defined additive character
$\psi_e:A/\fp\to E[\fp], a+\fp \mapsto ae$. On the other hand, every additive character $\psi:A/\fp\to E[\fp]$ is of that form for $e=\psi(1+\fp)$, since $A/\fp$ is generated as an $A$-module by $1+\fp$.
\end{Example}

\begin{Lemma}\label{projection-on-gauss-space}
Let $\chi$ be a multiplicative character.
\begin{enumerate}[label=$(\arabic*)$]
\item \label{item:a} The map
\[  \eta_{\chi}: \Fp\otimes E[\fp]\longrightarrow \Fp\otimes E[\fp],\quad g\longmapsto
-\sum_{x\in (A/\fp)^{\times}} \left( \chi(x)^{-1}\otimes x \right)g \]
is an $\Fp$-linear projection onto $G(E,\chi)$.
\item \label{item:b} For $e\in E[\fp]$, the image $\eta_\chi(1\otimes e)$ is the Gauss-Thakur sum $g(\chi,\psi_e)$. Here $\psi_e$ is the additive character attached to $e$ as in Example \ref{psi-e}.
\item \label{item:c} Every Gauss-Thakur sum $g(\chi,\psi)$ lies in $G(E,\chi)$, and $G(E,\chi)$ is generated  as $\Fp$-vector space by the Gauss-Thakur sums.
\end{enumerate}
\end{Lemma}
\begin{proof}
Part \ref{item:b} is easily verified from the definitions, and \ref{item:c} directly follows from \ref{item:a} and \ref{item:b} using that every $\psi$ is of the form $\psi_e$, and that the elements of the form $1\otimes e$ generate the $\Fp$-vector space $\Fp\otimes E[\fp]$. Hence, it remains to prove \ref{item:a}: \\
$\Fp$-linearity is clear. For all $g\in \Fp\otimes E[\fp]$ and $a\in A$, one has
 \[  (1\otimes a)(\eta_{\chi}(g))=0=(\chi(a)\otimes 1)\eta_{\chi}(g), \]
 if $a\in \fp$, as well as
\begin{eqnarray*}  (1\otimes a)(\eta_{\chi}(g))&=& -\sum_{x\in (A/\fp)^{\times}} \left( \chi(x)^{-1}\otimes ax \right)g\\
&=& -\sum_{x\in (A/\fp)^{\times}} \left( \chi(ax)^{-1}\chi(a)\otimes ax \right)g
= (\chi(a)\otimes 1)\eta_{\chi}(g), \end{eqnarray*}
if $a\notin \fp$. Hence, the image lies in $G(E,\chi)$. Furthermore, for all $g\in G(E,\chi)$, one has
\begin{eqnarray*}  \eta_{\chi}(g) &=&  -\sum_{x\in (A/\fp)^{\times}} \left( \chi(x)^{-1}\otimes x \right)g
=  -\sum_{x\in (A/\fp)^{\times}} (\chi(x)^{-1}\chi(x)\otimes 1 )g \\
&=& - ((q^{\degp} -1)\otimes 1)\cdot g = g.\end{eqnarray*}
Hence, $\eta_{\chi}$ is a projection.
\end{proof}

\begin{Remark}
An \textit{isogeny} $f: (E,\phi)\rightarrow (E',\phi')$ of Anderson $A$-modules over $L$ of dimension $d$ is an $\bF$-vector space scheme morphism from $E$ to $E'$ with finite kernel over $\CI$ such that $f\circ \phi_a=\phi_a'\circ f$ for all $a\in A$. It is automatically surjective, as its image is a subgroup scheme of $E'$ of the same dimension. The \textit{degree} of $f$ is defined as the Fitting ideal of $\ker(f)$ with respect to its natural $A$-module scheme structure. For a prime $\fp$ of $A$ which does not divide the degree of $f$, we have an induced isomorphism of $A$-module schemes $E[\fp]\cong E'[\fp]$. Hence, for $\chi$ a multiplicative character of conductor $\fp$ not dividing $\deg(f)$, we naturally have $G(E,\chi)\cong G(E',\chi)$.  
\end{Remark}

If the lift $\chi:A\to \Fp$ of a multiplicative character is a homomorphism of $\bF$-algebras, we can extend it to a $\Fbar$-linear homomorphism $\Fbar\otimes A\to \Fbar$ whose kernel is some maximal ideal $\fP$ above $\fp$. The extended map is then just ``evaluation at $\fP$'', i.e. $\chi$ sends $a$ in $A$ to $a(\mathfrak{P})=1\otimes a\pmod{\mathfrak{P}}$.
Conversely, any maximal ideal $\mathfrak{P}$ above $\fp$ in $\overline{\bF}\otimes A$ defines a multiplicative character 
\[\chi_\fP:(A/\fp)^{\times}\longrightarrow \Fp^{\times}, \quad a+\fp\longmapsto a(\fP).\]
As those multiplicative characters play an important role, we make the following definition.
\begin{Definition}
A multiplicative character $\chi$ will be called \textit{evaluative}, if its lift is a homomorphism of $\bF$-algebras. 
When $\fP$ is the corresponding maximal ideal in $\Fbar\otimes A$, we will also denote $\chi$ by $\chi_\fP$. 

In particular, we will just speak of an evaluative character $\chi_\fP$ which implies that $\fP$ is the corresponding maximal ideal in $\Fbar\otimes A$.
\end{Definition}

\begin{Proposition}\label{nonzero} \label{zero}
Let $\chi$ be a multiplicative character.
\begin{enumerate}[label=$(\arabic*)$]
\item \label{item:A} If $\chi$ is not evaluative, then $G(E,\chi)$ is zero. In particular, for any additive character $\psi$, $g(\chi,\psi)$ is zero.
\item \label{item:B} If $\chi$ is evaluative and $\{\psi_i\}_i$ are $\Fp$-linearly independent additive characters in the space $\Hom_A(A/\fp,E[\fp])$, then their Gauss-Thakur sums $\{g(\chi,\psi_i)\}_i$ are $\Fp$-linearly independent in $G(E,\chi)$.
\end{enumerate}
\end{Proposition}
\begin{proof}
For part \ref{item:A}, we prove that if $G(E,\chi)$ is non-zero, then $\chi$ is evaluative. If $g$ is a non zero element in $G(E,\chi)$, for $x$ and $y$ in $A$, $\alpha\in \bF$
\begin{eqnarray*} ((\alpha \chi(x)+\chi(y))\otimes 1)g
&=& (\alpha \otimes x)g+(1\otimes y)g =1\otimes (\alpha x+y)g \\
&=& (\chi(\alpha x+y)\otimes 1)g.
\end{eqnarray*}
As $g$ was non-zero, the lift $\chi$ is $\bF$-linear and - as it comes from a multiplicative character - even an $\bF$-algebra morphism, i.e. the character is evaluative. 

For part \ref{item:B}, let $\chi$ be evaluative. We first prove that $g(\chi,\psi)$ is non zero if $\psi$ is non zero. By finite Fourier inversion formula -- which was already mentioned in \cite[Prop.~1]{thakur88} -- for all $y$ in $A/\fp$, one has 
\begin{equation}\label{fourier}
1\otimes \psi(y)=-(q^{d_{\fp}}-1)\sum_{\chi}{g(\chi,\psi)(\chi(y)\otimes 1)}
\end{equation}
where the sum runs over multiplicative characters. By part \ref{item:A}, we can assume that the above sum runs over the set of evaluative $\chi$. This set is finite with $d_{\fp}$ elements and $\Frob_{\bF}$ as an element of $\Gal(\Fp/\bF)$ acts transitively on it by right-composition. Hence, if one of the $g(\chi,\psi)$ is zero, all are, and (\ref{fourier}) implies that $\psi$ is zero. 

Now, let $\{\psi_i\}_i$ be $\Fp$-linearly independent elements in $\Hom_A(A/\fp,E[\fp])$. We have a decomposition $E[\fp]=\bigoplus_i \operatorname{im}(\psi_i)$. Hence, if one has some %linear dependence
relation $0=\sum_{i}{(\alpha_i\otimes 1)g(\chi,\psi_i)}$, by taking the projection onto $\operatorname{im(\psi_i)}$ with respect to this decomposition, one is led to $(\alpha_i\otimes 1)g(\chi,\psi_i)=0$. As $g(\chi,\psi_i)$ is non zero, $\alpha_i$ is zero.
\end{proof}

To any element of the period lattice $\Lambda_E$, one can associate a specific Gauss-Thakur sum according to the next definition.
\begin{Definition}\label{additive}
Let $u_{\fp}$ be a uniformizer at $\fp$ in $A$, and choose an element $z_\fp$ of the fractional ideal $\fp^{-1}$ which satisfies $\up z_\fp\equiv 1 \mod \fp$. For $\lambda$ an element of the period lattice $\Lambda_E$, one obtains an additive character $\psi_\lambda$ given by
\[  \psi_\lambda:A/\fp \longrightarrow \fp^{-1}A/A \longrightarrow E[\fp], \quad a\longmapsto az_{\fp}\longmapsto \exp_E(\partial (a z_\fp)(\lambda)). \]
It defines a map $\Lambda_E \rightarrow \Hom_A(A/\fp,E[\fp]),~\lambda\mapsto \psi_\lambda$ and, by composition with $\psi \mapsto g(\chi,\psi)$, an $A$-module morphism
\begin{equation}
g_{u_{\fp}}:\Lambda_E\longrightarrow G(E,\chi). \nonumber
\end{equation}
\end{Definition}

\begin{Remark}
In the definition, we could have omitted the uniformizer $\up$, and just worked with a generator $z_{\fp}$ of $\fp^{-1}A/A$. In Section \ref{sec:values}, however, we will need a uniformizer, and therefore decided to put the dependence on $\up$ already here.

It should also be noted that after fixing $\up$, the definitions of $\psi_\lambda$ and of $g_{\up}$ do not depend on $z_{\fp}$, as long as $\up z_\fp\equiv 1 \mod \fp$. Namely, for a second such element $z'_{\fp}$, one has $\up(z_\fp-z'_{\fp})\equiv 0 \mod \fp$. Hence,
$z_\fp-z'_{\fp}\in A$ which implies 
\[ \exp_E(\partial z_\fp(\lambda))=\exp_E(\partial z'_\fp(\lambda)) \quad \forall \lambda\in \Lambda_E. \]
\end{Remark}
The following proposition provides a criterion of uniformizability.
\begin{Proposition}\label{uniformizable}
$E$ is uniformizable if, and only if, $g(\chi,\psi_{\lambda})$ generates $G(E,\chi)$ for $\lambda$ running through generators of $\Lambda_E$ as an $A$-module.
\end{Proposition}
\begin{proof}
The period lattice $\Lambda_E$ is a finitely generated and torsion free module over $A$, consequently projective of some rank $s$. Torsion-free is clear, and finite generation follows from the $\bF[t]$-case in \cite[Lem.~2.4.1]{anderson}. Further, $s\leq r$ with equality if, and only if, $E$ is uniformizable. In any case, the exponential map induces an injection of $A$-modules $\Lambda_E/\fp\Lambda_E \hookrightarrow E[\fp]$. We then have a well-defined injection of $\bF$-vector spaces
\begin{equation}
\Lambda_E/\fp\Lambda_E \longrightarrow \Hom_A(A/\fp,E[\fp]), \quad \lambda+\fp\Lambda_E \longmapsto \psi_{\lambda} \nonumber
\end{equation}
which is bijective if, and only if, $E$ is uniformizable by comparing dimensions. We conclude with Lemma \ref{projection-on-gauss-space} \ref{item:c}.
\end{proof}

\begin{Remark}
Proposition \ref{nonzero} together with Proposition \ref{uniformizable} and its proof even yield: If $\chi$ is evaluative, then $E$ is uniformizable if, and only if, the induced map $\Lambda_E/\fp \Lambda_E \rightarrow G(E,\chi)$ is an isomorphism of $\Fp$-vector spaces. 
\end{Remark}

\begin{Remark}
We assume here that $L$ is a finite separable extension of $K$. Let $K^{\text{sep}}$ be the separable closure of $K$ in $\CI$ and let $G_K$ be the Galois group of the field extension $K\subset K^{\text{sep}}$. The action of $G_{K}$ on $E[\fp]$ defines, by left-composition on the the set of additive characters, a continuous action on $G(E,\chi)$. Namely, $g(\chi,\psi)^{\sigma}=g(\chi,\sigma\circ \psi)$ for $\sigma$ in $G_K$. In the case where $E$ is uniformizable, we have
\begin{equation}
T_{\fp}(E):=\varprojlim E[\fp^n]\cong \varprojlim \Lambda_E/\fp^n\Lambda_E\cong A_{\fp}^r \nonumber
\end{equation}
and it induces a continuous $\fp$-adic representation $\rho_{E,\fp}:G_K\to \GL_r(A_{\fp})$. If further $\chi$ is evaluative, the isomorphism $\Lambda_E/\fp \Lambda_E \rightarrow G(E,\chi)$ implies that the continuous representation $G_K$ on $G(E,\chi)$ is isomorphic to $\chi\circ \rho_{E,\fp}$.
\end{Remark}

%Values of omega in terms of TGTS
%\input{Text/Values}
\section{Values of special functions at algebraic points}\label{sec:values}
Let $\mathfrak{P}$ be a maximal ideal in $\overline{\bF}\otimes A$ above $\fp$. For $f$ in the general Tate algebra $\bT$ and $\chi_{\mathfrak{P}}$ the multiplicative character corresponding to evaluation at $\mathfrak{P}$, we denote $f(\mathfrak{P})$ in $\Fp\otimes \CI$ the image of $f$ under $(\chi_{\mathfrak{P}} \hat{\otimes} 1)$. Similarly, for $\omega$ in $E(\bT)$, we define $\omega(\mathfrak{P})$ in $\Fp\otimes E(\CI)$ as the image of $\omega$ through $(\chi_{\mathfrak{P}} \hat{\otimes} \id)$.

\begin{Proposition}\label{prop}
For any special function $\omega$, the value $\omega(\mathfrak{P})$ is in $G(E,\chi_{\mathfrak{P}})$. In particular by Lemma \ref{projection-on-gauss-space} \ref{item:c}, $\omega(\mathfrak{P})$ is a linear combination of certain $g(\chi_{\mathfrak{P}},\psi)$ with coefficients in $\Fp$.
\end{Proposition}
\begin{proof}
Let $\omega\in \fsf(E)\subset E(\bT)$. It satisfies $(a\otimes 1)\omega=(1\otimes a)\omega$ for all $a\in A$. Composition with $(\chi_{\mathfrak{P}} \hat{\otimes} \id)$ yields $(a(\mathfrak{P})\otimes 1)\omega(\mathfrak{P})=(1\otimes a)\omega(\mathfrak{P})$ in $\overline{\bF}\otimes E(\CI)$ for all $a\in A$. 
\end{proof}

\begin{Theorem}\label{values}
Let $u_{\fp}$ be a uniformizer of $\fp$ in $A$. In the category of $A$-modules, the following diagram commutes: \\
\centerline{
\xymatrix@C+10mm{
\fd_{A/\bF[u_{\fp}]}\cdot \Lambda_E \ar[r]^{\delta_{u_{\fp}}} \ar[dr]_{g_{u_{\fp}}} &   \fsf(E) \ar[d]^{(\chi_{\mathfrak{P}} \hat{\otimes} \id)} \\
 & G(E,\chi_{\mathfrak{P}})
}}
where $\delta_{u_{\fp}}$ is the map of Theorem \ref{lattice} and $g_{u_{\fp}}$ is given by Definition \ref{additive}.
\end{Theorem}

\begin{Example}\label{example}
Let $C$ be $\bP^1_{\bF}$ and $E=\CC$ be the Carlitz module. Let $\infty$ be the closed point $[0:1]$ of $C$ and fix $t$ in $A$ such that $t^{-1}$ is a uniformizer at $\infty$. We have $A=\bF[t]$ and we let $\fp$ be a prime in $A$. Clearly, $A=\bF[t,u_{\fp}]$ for $u_{\fp}=\fp(t)$. Hence, by \cite[Prop.~III.2.4]{neukirch}, $\fd_{A/\bF[u_{\fp}]}=(\fp'(t))$. If $\zeta$ is a root of $\fp$ in $\overline{\bF}$, evaluation at $\zeta$ defines an evaluative multiplicative character $\chi_{\zeta}$.

By Proposition \ref{prop:delta-always-the-same}, $\delta_{\up}$ sends $\fp'(t)\cdot \tilde{\pi}$ in $\fd_{A/\bF[u_{\fp}]}\cdot \Lambda_{\CC}$ to the Anderson-Thakur function $\omega(t)$ in $\fsf(\CC)$.
Then, the commutativity of the diagram in Theorem \ref{values} is equivalent to Angl\`es and Pellarin's formula (see \cite[Thm.~2.9]{angles}):
\begin{equation}
\omega(\zeta)=\fp'(\zeta)g(\chi_{\zeta},\psi_{\tilde{\pi}}). \nonumber
\end{equation}
\end{Example}
\begin{Remark}
Since $\fd_{A/\bF[u_{\fp}]}$ corresponds, as an invertible $A$-module, to the inverse of the module of K\"ahler differentials of $A$ over $\bF$, we could have avoid the use of the different ideal thanks to a suitable tensor product with $\Omega_{A/\bF}^1$. The use of $\fd_{A/\bF[u_{\fp}]}$ however has the benefit that it allows a better understanding of the appearance of the factor $\fp'(\zeta)$ in Angl\`es and Pellarin's formula.
\end{Remark}

\begin{proof}[Proof of Thm.~\ref{values}]
As all maps are $A$-linear, it suffices to prove commutativity of the diagram for generators of 
$\fd_{A/\bF[\up]}\cdot \Lambda_E$, i.e. for elements of the form $f'_s(s)\lambda$ where $s\in A$ is a generator for the field extension $K/\bF(\up)$, and $f_s(X)\in \bF[\up][X]$ is its minimal polynomial (cf.~Section~\ref{characterisation-of-different}), as well as $\lambda\in \Lambda_E$. Further, let $\sum_i d_i\otimes a_i\in \fD_{A/\bF[\up]}\subseteq A\otimes_{\bF[\up]} A$ be the unique preimage of $f'_s(s)$ under the multiplication map. 

%With this choice, we have seen in Thm.~\ref{lattice} that $\delta_{\up}(f'_s(s)\lambda)\in \fsf(E)$, and in Thm.~\ref{prop} that $\delta_{u_{\fp}}(f'_s(s)\lambda)(\fP)\in G_\otimes(\chi_\fP)$.

By definition of $\delta_{\up}$ (see Theorem~\ref{lattice}), we have
\begin{eqnarray*} \delta_{u_{\fp}}(f'_s(s)\lambda)&=& (\id \hat{\otimes}\exp_E)\left((\id\otimes \partial u_{\fp}-u_{\fp}\otimes 1)^{-1}\bigl(\sum_{i} (d_i\otimes \partial a_i)\lambda\bigr)\right) \\
&=& \sum_{i}{(d_i\otimes 1)(\id \hat{\otimes}\exp_E)\left((\id\otimes \partial u_{\fp}-u_{\fp}\otimes 1)^{-1}(1\otimes  \partial a_i \lambda)\right)}.
\end{eqnarray*}
Once one evaluates at $\chi_{\fP}$, as $\chi_{\fP}(u_{\fp})=0$, it yields 
\begin{equation}\label{calculus}
\delta_{\up}(f'_s(s)\lambda)(\fP)=\sum_i d_i(\fP)\otimes a_i\exp_E(\partial u_{\fp}^{-1}\lambda). 
\end{equation}
By Proposition \ref{prop}, we know that the latter lands in $G(E,\chi_{\fP})\subseteq \Fp\otimes E[\fp]$. However, the summands $a_i\exp_E(\partial u_{\fp}^{-1}\lambda)$ don't have to lie in $E[\fp]$. To overcome this issue, we choose elements $c_j$ in $A$ such that $\{c_j(\fP)\}_j$ is an $\bF$-basis of $\Fp$, and rewrite
$\sum_i d_i(\fP)\otimes a_i\in \Fp\otimes A$ as $\sum_j c_j(\fP)\otimes b_j$ for appropriate $b_j\in A$.\footnote{The reader should notice that in general $\sum_i d_i\otimes a_i\neq \sum_j c_j\otimes b_j$. The two sums only agree modulo $\fp\otimes A$.}
Accordingly, we rewrite the expression for $\delta_{\up}(f'_s(s)\lambda)(\fP)$, as
\[ \delta_{\up}(f'_s(s)\lambda)(\fP) = \sum_j c_j(\fP)\otimes b_j\exp_E(\partial \up^{-1}\lambda)
= \sum_j \chi_\fP(c_j)\otimes b_j\exp_E(\partial \up^{-1}\lambda). \]
Since the $c_j(\fP)$ build an $\bF$-basis of $\Fp$, we obtain that $b_j\exp_E(\partial \up^{-1}\lambda)$ is in $E[\fp]$ for all~$j$. As the last statement is independent of $\lambda$, we conclude that $\up^{-1}b_j\in \fp^{-1}A$ for all~$j$.

By applying the projection operator $\eta:=\eta_{\chi_{\mathfrak{P}}}$ of Lemma \ref{projection-on-gauss-space}, we obtain
\begin{eqnarray*} \delta_{\up}(f'_s(s)\lambda)(\fP)&=& \eta\left( \delta_{\up}(f'_s(s)\lambda)(\fP) \right) \\
&=& -\sum_{y\in (A/\fp)^\times}
 \sum_j \chi_{\fP}(y)^{-1}\chi_{\fP}(c_j)\otimes y b_j \exp_E(\partial \up^{-1}\lambda) \\
 &=& -\sum_j \sum_{x\in (A/\fp)^\times}
 \chi_{\fP}(x)^{-1}\otimes xc_jb_j \exp_E\left(\partial \up^{-1}\lambda\right) \\
 &=& - \sum_{x\in (A/\fp)^\times}
 \chi_{\fP}(x)^{-1}\otimes x \exp_E\left(\partial \up^{-1}\partial\left(\sum_j c_jb_j\right)\lambda\right).
 \end{eqnarray*}
Take into account that the change of variable in the third line is valid, since $b_j \exp_E(\partial \up^{-1}\lambda)\in E[\fp]$.
Finally, 
 \[ \sum_j (c_jb_j)\equiv \sum_i d_i a_i=f'_s(s) \mod \fp, \]
and therefore, $\up^{-1}\sum_j (c_jb_j)$ and $z_\fp f'_s(s)$ reduce to the same element in $\fp^{-1}A/A$, where $z_{\fp}$ is an element as in Definition \ref{additive}.  Hence, $\exp_E(\partial (u_{\fp}^{-1}\sum_j (c_jb_j))\lambda)=\psi_{\lambda}(f'_s(s))
 =\psi_{f'_s(s)\lambda}(1)$, and
 \[ \delta_{\up}(f'_s(s)\lambda)(\fP) = - \sum_{x\in (A/\fp)^\times}
 \chi_{\fP}(x)^{-1}\otimes \psi_{f'_s(s)\lambda}(x) = g_{\up}(f'_s(s)\lambda). \qedhere \]
 \end{proof}

%Calculations for Carlitz Tensor Powers
%\input{Text/Carlitz}

%Calculations for Drinfeld 
%\input{Text/Drinfeld}

%Application to special L values
%\input{Text/LFunctions}
\section{Special Values of Goss $L$-functions} \label{L-values}
We discuss in this last section a question of Thakur in \cite{thakur91}, about whether Gauss-Thakur sums enter into the theory of $L$-functions developed by Goss. According to a general trend, special functions of rank $1$ Anderson $A$-modules are linked to special $L$-values via Pellarin type formulas (see for instance \cite{pellarin}, \cite{tuan} or \cite{green}). As opposed to the number field setting, the appearance of Gauss-Thakur sums in values of $L$-functions should then be explained by Theorem \ref{values} rather than by a functional equation. We clarify this thought in two situations: when $E$ is the Carlitz-module or when $A$ is the coefficient ring of an elliptic curve and $E$ is a Drinfeld-Hayes module.

We warn the reader that we use in this subsection a slight abuse of notation: as we want to remove the tensor from our notation to have lightened formulas, we will implicitly give our results under the multiplication map $m:\overline{\bF}\otimes \CI\to \CI$ without referring to $m$ anymore.

\subsection{The case $E=\CC$}
Assume that $A=\bF[t]$ for $t$ in $A$ and fix another indeterminate $\theta$. According to Pellarin in \cite{pellarin}, one has the identity
\begin{equation}
\sum_{a\in A^{+}}{\frac{a(t)}{a(\theta)}}=-\frac{\tilde{\pi}}{(t-\theta)\omega(t)}, \nonumber
\end{equation}
where $A^+$ corresponds to the set of monic polynomials. Applying the $q$-Frobenius map in the variable $\theta$ followed by the assignation $t=\zeta$, a root of unity with minimal polynomial $\fp$ over $\bF$, \cite[Thm.~2.9]{angles} yields  
\begin{equation}\label{function}
L(\chi_{\zeta},q):=\sum_{a\in A^+}{\frac{a(\zeta)}{a(\theta)^q}}=-\frac{\tilde{\pi}^q}{(\zeta-\theta^q)(\zeta-\theta)\mathfrak{p}'(\zeta)g(\chi_{\zeta},\psi_{\tilde{\pi}})}.
\end{equation}

A similar identity can be derived for classical $L$-functions: Let $p$ be a prime number and $\chi$ a primitive character of $(\bZ/p\bZ)^{\times}$. Then, using the functional equation of $L(\chi,s)$, one easily derives 
\begin{equation}\label{number}
L(\chi,2)=\left(\sum_{n=1}^p{\frac{n^2}{p^2}}\right)\frac{\pi^2}{g(\bar{\chi})},
\end{equation}
where $g(\bar{\chi})$ is the Gauss sum attached to $\bar{\chi}$. The high level of similarity between (\ref{function}) and (\ref{number}) is of a remarkable charm, since it is unlikely that the Goss $L$-function satisfies a functional equation.

\subsection{The case where $C$ is an elliptic curve}
We now give the analogous result (\ref{function}) when $(C,\cO_C)$ is an elliptic curve over $\bF$ thanks to the work of Green and Papanikolas in \cite{green} and our Theorem \ref{values}. Let $(C,\cO_C)$ be an elliptic curve over $\bF$ and let $\infty$ be its origin. Then, $A=H^0(C\setminus \{\infty\},\cO_C)$ can be written as the coefficient ring $\bF[t,y]$ where $t$ and $y$ are solutions of a Weierstrass equation 
\begin{equation}
W(y,t):=y^2+a_1ty+a_3y-(t^3+a_2t^2+a_4t+a_6)=0 \quad (a_i\in \bF). \nonumber
\end{equation}
In particular, $A$ admits $\{t^i,t^j y~ |~ i\geq 0,~j\geq 0\}$ as an $\bF$-basis, and by Proposition \ref{proptate}, $\bT$ coincides with the affinoid algebra considered in \cite{green}. With respect to this decomposition, the term of highest degree in the expansion of $a$ in $A$ is denoted $\sgn(a)$. It defines an application $A\setminus \{0\} \to \bF^{\times}$ which extends to a group morphism $K_{\infty}^{\times}\to \bF^{\times}$, where $K_{\infty}\subset \CI$ is the completion of $K$ at the place $\infty$. A non-zero element $a\in A$ is said to be \textit{monic} if it belongs to the set 
\begin{equation}
A_+:=\{a\in A \setminus\{0\}~ |~\sgn(a)=1\}. \nonumber
\end{equation}
Let $m\geq 1$ be an integer. The following series 
\begin{equation}
L(A,m):=\sum_{a\in A_+}{\frac{a\otimes 1}{(1\otimes a)^m}} \nonumber
\end{equation}
converges in $\bT$ to the \textit{Pellarin $L$-series at} $m$ of $A$. Let $\fp$ be a non-zero prime ideal of $A$, let $\mathfrak{P}$ be a maximal ideal in $\overline{\bF}\otimes A$ above $\fp$ and let $\chi_{\mathfrak{P}}$ be the evaluation at $\mathfrak{P}$. The value 
\begin{equation}
L(\chi_{\mathfrak{P}},m):=L(A,m)(\mathfrak{P})=\sum_{a\in A_+}{\frac{\chi_{\mathfrak{P}}(a)}{a^m}} \quad \text{in~} \CI \nonumber
\end{equation}
is referred to the \textit{Goss $L$-value of $\chi_{\mathfrak{P}}$ at} $m$. 

Let $H\subset \CI$ be the Hilbert class field of $K$. By a \textit{sign normalized} Drinfeld-Hayes $A$-module over $H$, we will mean a Drinfeld-Hayes $A$-module over $H$ which, as a $\bF$-vector space scheme, is equal to $\bG_{a,H}$ and where the $A$-action stems from the action of $t$ and $y$ given by
\begin{equation}
\phi_t=t+x_1\tau+\tau^2, \quad \phi_y=y+y_1\tau+y_2\tau^2+\tau^3 \quad \text{in~}\End_{\bF,H}(\bG_{a,H})=H\{\tau\}, \nonumber
\end{equation}
where $\tau$ is the Frobenius on $\bG_{a,H}$ and where $x_1$, $y_1$, $y_2$ are elements of $H$ such that $W(\phi_t,\phi_y)=0$.

We let $V=V_{\id}$ be a divisor on $C\times_{\bF}\Spec(H)$ as in Lemma \ref{divisor} for $\kappa=\id_{\bG_{a,L}}$. We define the \textit{unit disk} $\operatorname{Spm}\bT$ of $C\times_{\bF} \Spec(\CI)$ as the set
\begin{equation}
\operatorname{Spm}\bT:=\{\fm ~\text{maximal~ideal~of}~A\otimes \CI ~|~ \fm\bT\neq \bT\}. \nonumber
\end{equation}
The naming ``unit disk" follows from \cite[Cor.~2.2.13]{bosch}. In \cite{green}, $V$ is chosen so that it is not supported on $\operatorname{Spm}\bT$. We prove the following Theorem:
\begin{Theorem}\label{special-L-values-elliptic}
Let $E$ be a sign-normalized Drinfeld-Hayes $A$-module over $H$ for which $V$ is not supported on $\operatorname{Spm}\bT$. 
For any non-zero element $\lambda$ in $\Lambda_E\subset \CI$, and for all non negative integers $n$, we have
\begin{equation}
L(\chi_{\mathfrak{P}},q^n)\in \Fp H \cdot \frac{\lambda^{q^n}}{g(\chi_{\mathfrak{P}},\psi_{\lambda})}. \nonumber
\end{equation}
\end{Theorem}
\begin{proof}
One easily deduces from \cite[Thm.~4.6,~Prop.~4.3,~Prop.~7.1]{green} the following:
\begin{enumerate}
\item the period lattice $\Lambda_E$ of $E$ and the special functions $\fsf(E)$ are free of rank $1$ over $A$, and we write $\Lambda_E=\tilde{\pi}_E A$,
\item $\omega_E:=\delta_t(\mu'(t)\otimes\tilde{\pi}_E)$ in $\fsf(E)$, where $\mu$ is the minimal polynomial of $y$ over $\bF(t)$, is an invertible element of $\bT$ and thus generates $\fsf(E)$ as an $A$-module,
\item there is an (explicit) element $h_E$ of $A\otimes H$ such that
\begin{equation} \label{gpp}
L(A,1)=\frac{h_E}{f_{E}} \frac{\tilde{\pi}_E}{\omega_E}
\end{equation}
where $f_{E}$ is the shtuka function of $E$ associated to $\kappa=\id_{\bG_{a,H}}$.
\end{enumerate}
Without loss of generality, we prove our statement for $\lambda=\tilde{\pi}_E$. We have $L(A,1)^{(n)}=L(A,q^n)$ and hence, $L(\chi_{\mathfrak{P}},q^n)=L(A,1)^{(n)}(\mathfrak{P})$. By formula (\ref{gpp}), $\omega_E L(A,1)^{(n)}= h_E^{(n)}/f_E f_E^{(1)}\cdots f_E^{(n)}$ is in $\bT$ and therefore $(h_E^{(n)}/f_E f_E^{(1)}\cdots f_E^{(n)})(\mathfrak{P})$ is well-defined in $\Fp H$. As $\omega_E$ is an invertible element in the Tate algebra, $\omega_E(\mathfrak{P})\neq 0$. Theorem \ref{values} yields
\begin{equation}
g(\chi_{\mathfrak{P}},\psi_{\tilde{\pi}_E})L(\chi_{\mathfrak{P}},q^n)\in \Fp \cdot \omega_E(\mathfrak{P})L(A,1)^{(n)}(\mathfrak{P}) \subset \Fp H\cdot \tilde{\pi}_E^{q^n}. \nonumber
\end{equation}
Finally, by \cite[Sect.~1]{thakur91}, $g(\chi_{\mathfrak{P}},\psi_{\tilde{\pi}_E})$ seen as a tensorless sum in $\CI$ is non-zero.
\end{proof}

\addcontentsline{toc}{section}{References}
\bibliographystyle{alpha}
\bibliography{bibliography}

\end{document}